\documentclass[10pt,a4paper,reqno,twoside]{amsart}
\usepackage{amsmath}
\usepackage{amssymb}
\usepackage{amsthm}
\usepackage{mathrsfs}
\usepackage{color}
\usepackage{amsmath}
\usepackage{amsfonts}
\usepackage{amsthm}
\renewcommand{\doteq}{{\mathrm{\,:=\,}}}

\newcommand{\xii}{{|\xi|}}

\newcommand{\supp}{\operatornamewithlimits{supp}}

\newcommand{\R}{{\mathbb R}}

\newcommand{\Z}{{\mathbb Z}}
\newcommand{\N}{{\mathbb N}}

\newtheorem{theorem}{\bf Theorem}[section]
\newtheorem{lemma}{\bf Lemma}[section]
\newtheorem{proposition}{\bf Proposition}[section]
\newtheorem{corollary}{\bf Corollary}[section]

\theoremstyle{remark}
  \newtheorem{remark}{\sc Remark}[section]
\theoremstyle{definition}
  
  \newtheorem{example}{\sc Example}[section]
 \newtheorem{notation}{Notation}
 
 \numberwithin{equation}{section}

\begin{document}

\title[The critical exponent for semilinear  evolution equations]{The move from  Fujita  to   Kato  type exponent for a class of semilinear evolution equations with %scale-invariant
time-dependent damping}
%{Critical Exponents - From  Kato to Fujita type- for a class of semilinear evolution equations with scale-invariant time-dependent damping}

\author{ Marcelo Rempel Ebert, Jorge Marques, Wanderley Nunes do Nascimento}

\address{Marcelo Rempel Ebert, Departamento de Computa\c{c}\~ao e Matem\'atica, Universidade de S\~ao Paulo, Ribeir\~ao Preto, SP, 14040-901, Brasil}
\address{Jorge Marques
 CeBER and FEUC, University of Coimbra, Av. Dias da Silva 165,
3004-512 Coimbra, Portugal }
\address{Wanderley Nunes do Nascimento, Departamento de Matem\'atica Pura e Aplicada, Universidade Federal do Rio Grande do Sul, RS,  91509-900 , Brasil}

 \begin{abstract}
In this paper, we derive suitable optimal $L^p-L^q$ decay estimates,
$1\leq p\leq 2\leq  q\leq \infty$, for the solutions to the
$\sigma$-evolution equation, $\sigma>1$, with scale-invariant
time-dependent damping  and power
nonlinearity~$|u|^p$,
\[ u_{tt}+(-\Delta)^\sigma u + \frac{\mu}{1+t} u_t=
|u|^{p}, \quad   t\geq0, \quad x\in\R^n,
\]
where~$\mu>0$, ~$p>1$. The  critical exponent $p=p_c$  for the
global (in time) existence of small data  solutions to the Cauchy
problem is related to the long time behavior of solutions, which
changes accordingly $\mu \in (0, 1)$  or $\mu>1$.
Under the assumption of small  initial data in $L^1\cap L^2$,  we
find the critical exponent
\[
p_c=1+ \max \left\{\frac{2\sigma}{[n-\sigma+\sigma\mu]_+}, \frac{2\sigma}{n}  \right\}
=\begin{cases}
1+ \frac{2\sigma}{[n-\sigma+\sigma\mu]_+}, \quad \mu \in (0, 1)\\
1+ \frac{2\sigma}{n}, \quad \mu>1.
\end{cases}
\]
  For $\mu>1$ it  is well known as Fujita type exponent, whereas for $\mu \in (0, 1)$
one can read it as a shift of Kato exponent.
  %The novel idea in this paper is to deal with small
%value of the parameter $\mu$ by applying the stationary phase method
%to get $L^1-L^\infty$ estimates for solutions to the associate
%linear Cauchy problem.
\end{abstract}



\maketitle

 \section{Introduction}

 In this paper we study the global (in time) existence of small
data solutions to the Cauchy problem for the  semilinear damped
$\sigma$-evolution equations with scale-invariant time-dependent
damping
\begin{equation}\label{eq:DPE}
\begin{cases}
u_{tt}+ (- \Delta)^\sigma u + \frac{\mu}{1+t} u_t =f(u), & t\geq0, \ x\in \mathbf{R}^n,  \\
u(0,x)=0, & x\in \mathbf{R}^n,\\
u_t(0,x)=u_1(x), & x\in \mathbf{R}^n,
\end{cases}
\end{equation}
where   $\mu>0$, $\sigma > 1$ and $f(u)=|u|^p$ for some $p>1$. The nonlinearity may have several shapes, for instance, the derived  results in this paper also hold if  $f(u)=|u|^{p-1}u$ or if
$f$ is locally Lipschitz-continuous satisfying \cite{DLR}
\[f(0)=0, \quad |f(u)-f(v)| \leq C |u-v|(|f(u)|^{p-1}+|f(v)|^{p-1}),\]
for some $p>1$. The important information is that the nonlinearity is a perturbation which may create blow-up in finite time, well known in the literature as source nonlinearity. If the  initial condition $u(0, x)$ is small, then
$f(u)$ becomes small for large $p$. For this reason one is often able to prove such a
global (in time) existence result only for some $p>p_c$.

Let us introduce some
previous results to the  Cauchy problem for the  semilinear free
$\sigma-$evolution equations
\begin{equation}\label{eq:classicNint}
\begin{cases}
u_{tt}+ (- \Delta)^\sigma u  =|u|^p,   \\
u(0,x)=0, \quad
u_t(0,x)=u_1(x).
\end{cases}
\end{equation}
We begin with results for
$\sigma=1$. If $1<p<{p_K}(n)=\frac{n+1}{[n-1]_+}$ Kato \cite{K}
proved the nonexistence of global generalized solutions to
\eqref{eq:classicNint}, for  small initial data with compact
support. On the other hand, John \cite{J} showed that $p=1+ \sqrt 2$
is the critical exponent for the global existence of classical
solutions with small initial data in space dimension $n=3$. A bit
later, Strauss \cite{Strauss} conjectured that the critical exponent
$p_S(n)$, $n\geq2$, is the positive root of
\[(n-1)p^2-(n+1)p-2=0 .\]
Glassey (\cite{G1}, \cite{G2}) solved
this conjecture for classical solutions in space dimension
$n=2$. However, in space dimensions $n>3$, Sideris \cite{Si} proved the
nonexistence of global solutions in $C([0,\infty)\times
L^{\frac{2(n+1)}{n-1}})$ with arbitrarily small
initial data for~$1<p< p_S(n)$, even under the assumption of spherical
symmetry. Later in the supercritical case~$p>
p_S(n)$, Lindblad and Sogge \cite{LS}(see
references therein for further reported results) proved a global
existence result up to $n\leq 8$ and for all $n$ in the case of
radial initial data (see also \cite{Kubo} for the case of odd space dimension).
In \cite{GLS}, the authors removed the assumption of spherical symmetry.\\
Then, for $\sigma>1$ and for space dimensions $1 \leq n \leq
2\sigma$,   in \cite{EL} it was obtained the critical exponent to
\eqref{eq:classicNint},
${p_K}(n)=\frac{n+\sigma}{[n-\sigma]_+}$, which
is of Kato type.

In ~\cite{TY}, the authors proved global existence of small data solutions for the semilinear damped wave equation
%
%\begin{equation}\label{eq:dampedwave}
\[
 u_{tt}-\Delta u + u_t = |u|^p, \qquad u(0,x)=u_0(x), \qquad u_t(0,x)= u_1(x),\]
 %\end{equation}
%
in the supercritical range~$p>1+2/n$, by assuming   small initial data with compact
support from the energy space. A previous existence result in space
dimensions $n = 1$
and $n = 2$ was proved in \cite{Ma}. The compact support assumption
on the initial data can be weakened. By
only assuming initial data in Sobolev spaces, the existence result
was proved in space dimensions $n = 1$ and $n = 2$ in~\cite{IMN04}, by using
energy methods, and in space dimensions ~$n\leq5$
in~\cite{N04}, by using~$L^r-L^q$ estimates, $1\leq r\leq q\leq
\infty$. Nonexistence of the global small data solution is proved
in~\cite{TY} for~$1<p< 1+2/n$ and in \cite{Z} for $p= 1+2/n$. The
critical case for more general nonlinearities has been recently
discussed in \cite{EGR}. The exponent~$p_{F}(n)\doteq 1+2/n$ is well
known as Fujita exponent and it is the critical index for the
semilinear parabolic problem~\cite{F66}:
\[
v_t -\triangle v = v^p\,, \qquad v(0,x)=v_0(x)\geq0\,.
\]
The diffusion phenomenon between linear heat and linear classical
damped wave models (see \cite{HM},  \cite{MN03},  \cite{N04} and
\cite{N03}) explains the parabolic
nature of classical damped wave
models with power nonlinearities from the point of view of decay estimates of solutions.\\
In \cite{W07} the author considered
a more general model
 \[u_{tt}-\Delta u + b(t)u_t = 0, \qquad u(0,x)=u_0(x), \qquad u_t(0,x)= u_1(x),
  \]
  with   a class of  time dependent  damping $b(t)u_t$ for which the critical
  exponent is still Fujita exponent $1+2/n$
  for the associate semilinear Cauchy problem with power nonlinearity
   $|u|^p$ (see \cite{DAL13} and \cite{DLR}).

We state now well known results for the semilinear  wave equation with scale-invariant time-dependent
damping
\begin{equation}\label{waveeq}
\begin{cases}
u_{tt}- \Delta u + \frac{\mu}{1+t} u_t =|u|^p,   \\
u(0,x)=u_0(x), \quad
u_t(0,x)=u_1(x).
\end{cases}
\end{equation}
This model is critical, in the sense that it is relevant the size of
the parameter $\mu$  to describe
the asymptotic behavior of solutions.
If $\mu\geq \frac53$ for $n=1$ or  $\mu\geq 3$ for $n=2$, by
assuming initial data in the energy spaces with additional
regularity $L^1(R^n)$, a global (in time) existence result for
(\ref{eq:DPE}) was proved in \cite{DAmmas} for $p> p_{F}(n)\doteq 1
+ \frac2n$.  %, the well known Fujita index \cite{F66}.
This result was extended by same author for higher space dimensions
$n\geq 3$ by assuming initial data in spaces with weighted norms for
$\mu\geq n+2$. The exponent $p_{F}$ is critical for this model, that
is, for $1 < p\leq p_{F}$ and suitable, arbitrarily small initial
data, there exists no global weak solution \cite{DAL13}.  In
\cite{DALR} the authors studied the special case $\mu = 2$ and
showed that the critical exponent for \eqref{waveeq} is given by
$p_c=\max\{p_{S}(n+2),  p_{F}(n)\}$. In the same paper the authors
also conjectured that  $p_c\geq \max\{p_{S}(n+\mu),  p_{F}(n)\}$ for
$\mu\in (2, n+2)$. The threshold value $\mu_{\star}$ is the solution
to $p_{S}(n+\mu_{\star})=p_{F}(n)$ and it is given by
\[\mu_{\star}=\frac{n^2+n+2}{n+2}.\]
In \cite{IS}, for suitable initial data, the authors obtained blow-up in finite time and gave the upper bound for the
lifespan of solutions to \eqref{waveeq}
if   $1<p\leq p_{S}(n+\mu)$ with $\mu \in (0, \mu_{\star})$.
It is worth
noticing that if  $\mu \in [0, \mu_{\star})$, then $p_{F}(n)<
p_{S}(n+\mu)$.\\
As far as we know, it is still a open problem to prove global
existence of small initial data solutions  for $p> p_{F}(n)$ in the
cases $\frac43<\mu< \frac53$ for $n=1$, $2<\mu< 3$ for $n=2$, or
$\mu_{\star}<\mu< n+2$ for $n\geq 3$.\\

A  related model to \eqref{waveeq} is the semilinear wave equation with scale-invariant mass and dissipation
\[
\begin{cases}
u_{tt}- \Delta u + \frac{\mu}{1+t} u_t + \frac{m^2}{(1+t)^2} u  =|u|^p,   \\
u(0,x)=u_0(x), \quad
u_t(0,x)=u_1(x).
\end{cases}
\]
For results about  existence and non-existence of global (in time) small initial data solutions, we address the reader to \cite{NPR, P2018, PR, PR2} and the references therein.

%\textcolor{green}{If the initial data are in the energy spaces
%without additional regularity $L^1(R^n)$ the critical exponent is
%$p_c(n)=1 + \frac4n$. In \cite{EM} assuming $\mu \geq 2m$ with
%$n\geq 3$ and $\frac{n}{2}<m \leq 1 + p_c(n) $ the authors showed a
%global in time existence result for $p> p_c(n) $ when initial data
%belongs to the space $H^m(\mathbf{R}^n)\times
%H^{m-1}(\mathbf{R}^n)$. Moreover, they derived a global existence
%result for the non-singular wave equation in the Einstein de Sitter
%model.}

%If we remove the assumption that the initial data are in $L^1(R^n)$
%and we only assume that they are in the energy space, then the
%critical exponent to (\ref{eq:DPE}) for $\sigma = 1$ is modified
%into $ 1 + \frac4n$ and one  may lower the thresholds required for
%$\mu$ (see Theorem 4 and Theorem 6 in \cite{DAmmas}). In \cite{EM}
%the authors derived higher order energy estimates for solutions to
%the linear problem associated to (\ref{eq:DPE}) and discussed some
%benefits  in its applications to prove global existence results.

The main goals in this paper
are to derive $L^p-L^q$ estimates and energy
estimates for solutions  to the linear Cauchy
problem associated to (\ref{eq:DPE}) and to obtain
 the critical exponent for the  global
(in time) existence of small initial data
solutions to (\ref{eq:DPE}). We conclude that $\mu=1$ is  the
threshold for the asymptotic behavior of solution to (\ref{eq:DPE}),
it means that the critical exponent is a shift of Kato type exponent
$p_K(n+\sigma\mu)\doteq \frac{n+\sigma
+\sigma\mu}{[n-\sigma+\sigma\mu]_+}$ for $0 < \mu< 1$ and  of Fujita
type  $p_{F}(n, \sigma)\doteq 1+ \frac{2\sigma}{n}$  for $\mu>1$.

The plan of the paper is the following:
\begin{itemize}
\item in Section~\ref{sec:results}, we collect and discuss our main results;
\item in Section \ref{lplqestimates}, we derive the ~$L^p-L^q$ estimates for solutions to the associate linear Cauchy problem;
\item in Section~\ref{NL}, we apply the decay estimates previously derived to prove Theorems~\ref{theoremNL2} and~\ref{theoremN1} for the nonlinear problems~\eqref{eq:DPE};
    \item  in Section~\ref{test}, we apply the test function method to prove Proposition \ref{testfuntion};
\item in Appendix, we  include  some notations, well known estimates for multipliers    and properties of special functions
used to prove our results throughout the paper.
\end{itemize}

\section{ Main results}\label{sec:results}

 Our first result is for small $\mu$ and
 $\sigma>1$, it shows that the critical exponent is a shift of the Kato exponent,  unlike other  case $\sigma=1$ where it
appears a shift of Strauss exponent \cite{DALR, IS}. In the next theorem we are going to use the following notation
\begin{equation}\label{mu}
\mu_{\sharp}=\begin{cases}
\infty, \quad if  \ \mu\leq 2 -\frac{2n}{\sigma},\\
\frac1{2\sigma}\left( \sigma -n + \sqrt{9\sigma^2-10n\sigma+n^2}\right), \quad if \  \mu> 2 -\frac{2n}{\sigma}.
\end{cases}
\end{equation}
\begin{theorem}\label{theoremNL2}
Let $\sigma>1$, $1\leq n<\sigma$,
 $1-\frac{n}{\sigma}<\mu < \min \left\{ \mu_{\sharp} ; 1 \right\}$, with $\mu_{\sharp}$ as in \eqref{mu}  and $\mu\neq 2-\frac{2n}{\sigma}$.
 If
\begin{equation}\label{kato}
  1+ \frac{2\sigma}{n-\sigma+\sigma\mu}\doteq p_K(n+\sigma\mu)<p\leq 1 + \frac{2\sigma-\sigma\mu}{[2n-2\sigma+\sigma\mu]_+}\doteq
  q_1,
\end{equation}
 then there exists
$\epsilon>0$ such that for any initial data
\[  u_1 \in \mathcal A=L^2(\mathbf{R}^n)\cap L^1(\mathbf{R}^n), \qquad  ||u_1||_\mathcal{A}\leq \epsilon,\]
there exists a unique energy solution $u\in C([0, \infty),
H^\sigma(\mathbf{R}^n))\cap C^1([0, \infty), L^2(\mathbf{R}^n)) \cap
L^{\infty}([0, \infty)\times\mathbf{R}^n) $ to (\ref{eq:DPE}).
Moreover,  for $2\leq q\leq q_1$ the solution satisfies the
following estimates
\[
||  u (t, \cdot) ||_{L^q} \lesssim
(1+t)^{-\frac{n}{\sigma}\left(1-\frac1{q}\right)+1-\mu} ||
u_1||_{\mathcal A},
\]
%\[
%||  u (t, \cdot) ||_{L^2} \lesssim (1+t)^{-\frac{n}{2\sigma}+1-\mu} || u_1||_{L^1\cap L^2},
%\]
\[
||  u (t, \cdot) ||_{L^\infty} \lesssim
(1+t)^{-\min\left\{\frac{n}{\sigma}+\mu-1, \frac{\mu}{2}\right\}} ||
u_1||_{\mathcal A},
\]
and
\[
\|   u(t,\cdot)\|_{\dot{H}^\sigma}+ \|\partial_t u(t,\cdot)\|_{L^2} \lesssim  (1+t)^{- \frac{\mu}{2 }} || u_1||_{\mathcal A}, \quad \forall t\geq 0.
\]
%and for $\gamma\in (0, \sigma]$
%  \[
%  \|   u(t,\cdot)\|_{\dot{H}^\gamma}\lesssim (1+t)^{-\min\left\{\mu-1+\frac{n+2\gamma}{2\sigma}, \frac{\mu}{2 }\right\}} || u_1||_{L^1\cap L^2}, \quad \forall t\geq 0
%\]
\end{theorem}
\begin{remark} The condition $\mu < \min \left\{ \mu_{\sharp} ; 1 \right\}  $ implies that the range for $p$ in \eqref{kato} is not empty, i.e., if
$\mu> 2 -\frac{2n}{\sigma}$,  $\mu_{\sharp}$ is the positive root
of $\sigma \mu^2+(n-\sigma)\mu+ 2(n-\sigma)=0$. In particular, $\mu_{\sharp}\geq 1$ if $3n\leq 2\sigma$.
Moreover, %$p\leq 1 + \frac{2\sigma-\sigma\mu}{[2n-2\sigma+\mu\sigma]_+}$
$\mu= 2-\frac{2n}{ \sigma}\left(1-\frac1{q_1}\right)$ and $\mu<
2-\frac{2n}{ \sigma}\left(1-\frac1{q}\right)$ for all $q<q_1$. If
$\mu\leq 2-\frac{2n}{\sigma}$, then $q_1=\infty$ in (\ref{kato}).
\end{remark}
\begin{remark}
If $\mu = 2 - \frac{2n}{\sigma}$, under the assumptions of Theorem  \ref{theoremNL2}, it is possible to obtain global (in time) unique energy solutions to \eqref{eq:DPE}, however a logarithm term appears on the estimates  for the solutions.
\end{remark}
\begin{remark}
In the limit value of $\mu$, $\mu=0$, we get
$p_K(n)=\frac{n+\sigma}{[n-\sigma]_+}$, so  we recover the critical
index obtained in \cite{EL}.
\end{remark}
\begin{remark} Using the properties of the Hankel functions and the representation \eqref{Hankel}
 we conclude that
 \[g_{j,k}(t)=\int_{\R^n}\xii^{2j\sigma}\,|\partial_t^k\hat u(t, \xi)|^2 d\xi, \quad j+k\leq 1,
 \]
are continuous functions on $[0, \infty)$. The validity of the Fourier inversion formula implies $ u \in C([0, \infty),
H^\sigma(\mathbf{R}^n))\cap C^1([0, \infty), L^2(\mathbf{R}^n))$.
But the same argument can not be used to conclude the continuity of $h(t)=|| u (t, \cdot) ||_{L^\infty}$.
  \end{remark}
\begin{example}
For the plate equation $\sigma=2$ in one space dimension $n=1$, the conclusions of Theorem \ref{theoremNL2} hold for all   $p>1+ \frac4{2\mu-1}$ and $\frac12<\mu<1$.
\end{example}
%
%It is important to point out that in Theorem
%\ref{theoremN1} we enlarge the admissible range of the parameter
%$\mu$ if we compare with a result already
%obtained for $\mu\geq \frac{n}{\sigma}+2$ in \cite{Pa}.
The next result is an extension  of  Theorem  2 in \cite{DAmmas}
done for the case $\sigma=1$.
\begin{theorem}\label{theoremN1}
Let $\sigma>1$, $ n < 2\sigma$ and  $\mu > \max \left\{ \frac{n}{\sigma}+ \frac{2n}{n+2\sigma} ; 1 \right\}$, $\mu\neq
\frac{2n}{\sigma} $ and $\mu\neq \frac{n}{\sigma}+2$. If
\begin{equation}\label{fujita}
  1+ \frac{2\sigma}{n}<p\leq  \frac{n}{[2n-\sigma\mu]_+}\doteq \frac{q_0}{2}, \end{equation}
   then there exists
$\epsilon>0$ such that for any initial data
\[  u_1 \in \mathcal A=L^2(\mathbf{R}^n)\cap L^1(\mathbf{R}^n), \qquad  ||u_1||_\mathcal{A}\leq \epsilon,\]
there exists a unique energy solution $u\in C([0, \infty),
H^\sigma(\mathbf{R}^n))\cap C^1([0, \infty), L^2(\mathbf{R}^n))\cap
L^{\infty}([0, \infty)\times\mathbf{R}^n) $ to (\ref{eq:DPE}).
Moreover, for $2\leq q\leq q_0$ the solution satisfies the following
estimates
\begin{equation}\label{ineqnonlinear1ell}
||  u (t, \cdot) ||_{L^q} \lesssim (1+t)^{-\frac{n}{2\sigma}\left(1-\frac1q\right)} || u_1||_{\mathcal A}, \quad \forall t\geq 0,
\end{equation}
\begin{equation}\label{ineqnonlinear1ell2}
||  u (t, \cdot) ||_{L^\infty} \lesssim
(1+t)^{-\min\left\{\frac{n}{2\sigma}, \frac{\mu}{2 }\right\}} ||
u_1||_{\mathcal A}, \quad \forall t\geq 0,
\end{equation}
\begin{equation}\label{ineqnonlinearfrac}
  \|   u(t,\cdot)\|_{\dot{H}^\sigma}+ \| \partial_t u(t,\cdot)\|_{L^2}\lesssim (1+t)^{-\min\left\{\frac{n}{2\sigma}+1, \frac{\mu}{2 }\right\}} || u_1||_{\mathcal A}, \quad \forall t\geq 0.
\end{equation}
%and for $\gamma\in (0, \sigma]$
%  \begin{equation}\label{ineqnonlinearfrac}
%  \|   u(t,\cdot)\|_{\dot{H}^\gamma}\lesssim (1+t)^{-\min\left\{\frac{n+2\gamma}{2\sigma}, \frac{\mu}{2 }\right\}} || u_1||_{L^1\cap L^2}, \quad \forall t\geq 0
%\end{equation}
%and
%\begin{equation}\label{ineqnonlinear2ell}
%\| \partial_t u(t,\cdot)\|_{L^2} \lesssim (1+t)^{-\min\{\frac{n}{2\sigma}+1, \frac{\mu}{2}\}}  || u_1||_{\mathcal A}, \quad \forall t\geq 0.
%\end{equation}
%\begin{equation}\label{ineqnonlinear2ell}
%\|(- \Delta)^{\frac{j\sigma}{2}} \partial_t^k u(t,\cdot)\|_{L^2} \lesssim (1+t)^{-\min\{1, \frac{\mu}{2}\}} || u_1||_{L^2}, \quad \forall t\geq 0.
%\end{equation}
\end{theorem}
\begin{remark}\label{hypothesis} The condition $\mu > \frac{n}{\sigma}+ \frac{2n}{n+2\sigma}$ implies that the range for $p$ in \eqref{fujita} is not empty. Moreover, $q\leq q_0$, with $q_0$ defined by \eqref{fujita},    is
equivalent to $\mu\geq \frac{2n}{
\sigma}\left(1-\frac1{q}\right)$. If $\mu\geq
\frac{2n}{\sigma}$ then $q_0=\infty$ in \eqref{fujita}.
\end{remark}
%\begin{remark}\label{q0} Let $q_0=1 + \frac{\sigma\mu}{[2n-\sigma\mu]_+}$.
%Then $q\leq q_0$,    is
%equivalent to $\mu\geq \frac{2n}{
%\sigma}\left(1-\frac1{q}\right)$.
%\end{remark}
%\begin{remark}\label{no}
%It is useful to remark that if $q\leq \frac{n}{[n-\sigma]_+}$ then   $\max\left\{1 , \frac{n}{\sigma}\left(1-\frac1{q}\right)\right\}=1$ in Theorem \ref{theoremlinear} . Moreover,
%if $1<\mu\leq \min\{ 2, \frac{2n}{\sigma}\}$, then $q_0\leq \frac{n}{[n-\sigma]_+}$.
%\end{remark}
\begin{remark} The cases $\mu=1$, $\mu=
\frac{2n}{\sigma} $ and $\mu= \frac{n}{\sigma}+2$ can also be included in Theorem \ref{theoremN1}, but it appears an additional logarithm loss on the derived estimates for the solutions. Moreover, to obtain the result for  higher space dimension $n\geq 2\sigma$, one also  have to derive  $L^p-L^q$ estimates, with  $p\in [1, 2]$,  for solutions to the linear problem  at low frequencies and combine it with  the  already obtained estimates at high frequencies.
  \end{remark}
    \begin{example}
For the plate equation $\sigma=2$,  Theorem \ref{theoremN1} applies  for  $\mu>1$ and $\mu\neq \frac52$ if $n=1$, for $\mu>\frac53$ and $\mu\notin \{2;  3\}$ if $n=2$ and  for $\mu> 2+ \frac5{14}$ and $\mu\notin \{3; \frac72\}$ if $n=3$.
\end{example}

For the sake of simplicity, in the next two results  we restrict our analysis for integer $\sigma$. However, the test function method
was recently applied in \cite{DF} for a class of $\sigma-$ evolution operators with  non-integer $\sigma$.

First let us discuss into details the non-existence result  for the non-effective damping cases $0<\mu\leq 1$.
The proof of the next  result  can be obtained with a slightly change in the proof of  Theorem 1.5 in \cite{W}:
\begin{proposition}\label{testfuntion}
\label{thm:test0dec} Let~$\sigma\in\N$,
$0<\mu\leq 1$ and
\[ 1<p\leq p_K(n+\sigma\mu)\doteq \frac{n+\sigma +\sigma\mu}{[n-\sigma+\sigma\mu]_+}.\]
If $u_1 \in L^1(\mathbb{R}^n)$ such that
\begin{equation}\label{blowup}
\int_{\R^n} u_1(x)\, dx>0,\end{equation} then
there exists no global (in time) weak solution $u\in
L^p_{loc}([0,\infty)\times\R^n)$ to~\eqref{eq:DPE}.
\end{proposition}
\begin{remark}
The proof of Proposition \ref{testfuntion} also holds for   $\mu>
1$, however it is not optimal(see Proposition
\ref{testfuntionFujita}).
\end{remark}
Since
\[1+ \frac{2\sigma }{n}=\frac{n+\sigma +\sigma\mu}{n-\sigma+\sigma\mu}\]
is equivalent to $\mu=1$ and $p_K(n+\sigma\mu)<1+ \frac{2\sigma
}{n}$ for $\mu>1$, so Proposition \ref{testfuntion} is not the
counterpart of Theorem \ref{theoremN1} for $\mu>1$. Applying Theorem
2.2 in \cite{DAL13}, one may have the following  improvement of
Proposition \ref{testfuntion} and  the counterpart of Theorem
\ref{theoremN1} is obtained.
\begin{proposition}\label{testfuntionFujita}
\label{thm:test0dec}
Let~$\sigma\in\N$, $\mu> 1$ and
\[ 1<p\leq 1+ \frac{2\sigma }{n}.\]
If $u_1 \in L^1(\mathbb{R}^n)$ such that
\begin{equation}\label{blowup2}
\int_{\R^n} u_1(x)\, dx>0,\end{equation} then
there exists no global (in time) weak solution $u\in
L^p_{loc}([0,\infty)\times\R^n)$ to~\eqref{eq:DPE}.
\end{proposition}

\begin{remark}
From  Theorem \ref{theoremN1} and Proposition \ref{testfuntionFujita} we conclude that for $\mu>1$ the Fujita type index
\[p=1+ \frac{2\sigma }{n}\]
 is the critical exponent to \eqref{eq:DPE}, whereas for $0<\mu\leq1$, Theorem \ref{theoremNL2} and
Proposition \ref{testfuntion} implies that the Kato type index
 \[p_K(n+\sigma\mu)=\frac{n+\sigma +\sigma\mu}{[n-\sigma+\sigma\mu]_+} \]
  is the  critical exponent to \eqref{eq:DPE}.
\end{remark}
The next remark was suggested by Prof. M. D'Abbicco and says that  Proposition \ref{testfuntion} could also be obtained by applying Theorem 2.2 in \cite{DAL13}.
\begin{remark}
Let us  consider
\[ u_{tt} + (-\Delta)^\sigma u + \frac\nu{1+t}u_t = (1+t)^{2\gamma}\,|u|^p, \]
with~$\nu>0$. Applying Theorem 2.2 in~\cite{DAL13}, one may derive a nonexistence result for
\[ 1+\gamma < p \leq p_c = 1+\frac{2(1+\gamma)\sigma}n\,.\]
%
%In particular, for~$\nu=\mu\geq1$ and~$\gamma=0$ you get~$p_c=1+2\sigma/n$ in
%
%\[ u_{tt} + (-\Delta)^\sigma u + \frac\mu{1+t}u_t = |u|^p. \]
If ~$\mu\in[0,1)$,  applying the change of variable~$v=(1+t)^{1-\mu}u$, so that
\[ v_{tt}+(-\Delta)^\sigma v + \frac{2-\mu}{1+t}v_t = (1+t)^{(p-1)(1-\mu)}\,|v|^p. \]
Setting~$\nu=2-\mu$ and~$\gamma=(p-1)(1-\mu)/2$,  Theorem 2.2 in \cite{DAL13} implies the nonexistence of solutions if
\[ 1+(p-1)\frac{1-\mu}2 < p \leq 1+\frac{(2+(p-1)(1-\mu))\sigma}n  \]
The left-hand side is clearly true, due to~$1-\mu<2$, and the
right-hand side gives the condition for the
desired critical exponent:
\[ p \leq 1+\frac{(2+(p-1)(1-\mu))\sigma}n,  \quad
i.e. \quad
 p \leq \frac{n+\sigma +\sigma\mu}{[n-\sigma+\sigma\mu ]_+}. \]
\end{remark}

\section{$L^p-L^q$ estimates for solutions}\label{lplqestimates}
Let us consider the Cauchy problem for the linear $\sigma$-evolution equation
 with scale-invariant time-dependent damping
\begin{eqnarray} \label{mainlineq}
u_{tt} + (- \Delta)^\sigma u + \frac{\mu}{1+t}u_t
=0,\,\,\,u(s,x)=0,\,\,\,u_t(s,x)=u_1(x)
\end{eqnarray}
in $[0,\infty) \times \mathbb{R}^n$, with $s\leq t$, $\mu>0$ and
$\sigma > 1$.

Taking the partial Fourier transform
with respect to the $x$ variable in
\eqref{mainlineq} we obtain
\begin{eqnarray} \label{AfterFourTrans1}
\widehat{u}_{tt} +  |\xi|^{2\sigma}\widehat{u} + \frac{\mu}{1+t}
\widehat{u_t}  = 0,
\,\,\,\widehat{u}(s,\xi)=0,\,\,\,\widehat{u}_t(s,\xi)=\widehat{u}_1(\xi).
\end{eqnarray}
According to   \cite{Pa} and \cite{Wirth}, we have
the following representation for the solution to
\eqref{AfterFourTrans1} in terms of the Hankel functions
$H^{\pm}_{\rho}$:
\begin{proposition}
Assume that $u$ solves the Cauchy problem \eqref{mainlineq} for data
$u_1\in \mathbb{S}(\mathbb{R}^n)$. Then the Fourier transform $\hat
u(t, s, \xi)$ can be represented as
\[\hat u(t, s, \xi)= \psi(t,s,\xi) \hat u_1(\xi),\]
where the multiplier $\psi$ satisfies
\begin{equation}\label{Hankel}
i|\xi|^{j\sigma}\partial_t^k\psi(t,s,\xi)=\frac{\pi}{4}\frac{(1+t)^{\rho}}{(1+s)^{\rho-1}}|\xi|^{(k+j)\sigma}
\left|
                         \begin{array}{cc}
                           H^-_{\rho}\left( (1+s)|\xi|^{\sigma} \right) & H^-_{\rho-k}\left(
                           (1+t)|\xi|^{\sigma} \right) \\
                           H^+_{\rho}\left( (1+s)|\xi|^{\sigma} \right) &
                           H^+_{\rho-k}\left((1+t)|\xi|^{\sigma} \right) \\
                         \end{array}
                       \right|
\end{equation}
with $k+j=0,1$  and
\[\rho= \frac{1-\mu}{2}.\]
\end{proposition}

In order to derive estimates for $\widehat{u}$ and its derivatives,
we divide the extended phase space into zones to analyse the
behavior of the Hankel functions
$H^{\pm}_{\rho}$(see Lemma \ref{hankel} in Appendix):
\[ Z_{high} = \{ \xi ; |\xi| \geq 1 \} \quad and \quad Z_{low} = Z_1 \cup Z_2 \cup Z_3\]
where
\[Z_1 = \{ \xi ; (1+s)^{-1}\leq |\xi|^\sigma \leq 1 \} ; \,\,\,  Z_2 = \{ \xi ;(1+s)|\xi|^\sigma \leq  1 \leq (1+t)|\xi|^\sigma  \} ; \,\,\, Z_3 = \{ \xi ; (1+t)|\xi|^\sigma \leq 1 \}.  \]
We consider the cut-off function $\chi \in C^{\infty}(\mathbb{R}^n)$  with
$\chi(r)=1$ for $r\leq \frac12$ and $\chi(r)=0$ for $r\geq 1$ and define
\begin{eqnarray*}
&\chi_1(s,\xi)=1- \chi((1+s)|\xi|^\sigma), \\
&\chi_2(t, s,\xi)= \chi((1+s)|\xi|^\sigma)\left( 1 - \chi((1+t)|\xi|^\sigma)\right) , \\
&\chi_3(t, s,\xi)= \chi((1+s)|\xi|^\sigma) \chi((1+t)|\xi|^\sigma),
\end{eqnarray*}
such that $ \chi_1+\chi_2+\chi_3=1$. In the following we decompose the multiplier
\[m(t, s, \xi)=|\xi|^{(k+j)\sigma} \left|
                         \begin{array}{cc}
                           H^-_{\rho}\left( (1+s)|\xi|^{\sigma} \right) & H^-_{\rho-k}\left(
                           (1+t)|\xi|^{\sigma} \right) \\
                           H^+_{\rho}\left( (1+s)|\xi|^{\sigma} \right) &
                           H^+_{\rho-k}\left((1+t)|\xi|^{\sigma} \right) \\
                         \end{array}
                       \right|
                       \]
as $m=(1-\chi)m + \chi m$ and  $\chi m=m\sum \chi_i$ and estimate each of the summands $(1-\chi)m$ and $m_i\doteq
m\chi_i, i=1,2,3$:

%\begin{itemize}
%%%%%%%%%%%%%%%%%%%%%%%%%%%%%%%%%%%%%%%%%%%%%%

\subsection*{Considerations in $Z_{1}$:}

In $Z_1$ we may estimate
\[
 |\chi(|\xi|)\chi_1( s,\xi) m(t, s, \xi)| \lesssim (1+s)^{-1/2}(1+t)^{-\frac12} |\xi|^{\sigma(k+j-1)}
\]
so that
\[|\chi(|\xi|)\chi_1(s,\xi)|\xi|^{j\sigma}
\partial_t^k\psi(t,s,\xi)|\lesssim
(1+s)^{\frac{\mu}2}(1+t)^{-\frac{\mu}2}|\xi|^{\sigma(k+j-1)}.
\]
By using Haussdorff-Young inequality and H\"{o}lder inequality,
setting
\[ \frac1r = \frac1{q'}-\frac1{p'}=\frac1p-\frac1q, \]
for $1\leq p\leq 2\leq q\leq\infty$ and $k+j = 0,1$ one may estimate
\begin{align*}
\|
\mathfrak{F}^{-1}(\chi(|\xi|)\chi_1(s,\xi)|\xi|^{j\sigma}\partial_t^k
\psi(t,s,\xi))\ast u_1\|_{L^q}
    & \lesssim \|  \chi(|\xi|)\chi_1(s,\xi)|\xi|^{j\sigma}\partial_t^k   \psi(t,s,\xi)\hat u_1\|_{L^{q'}}\\
     &\lesssim \| \chi(|\xi|)\chi_1(s,\xi)|\xi|^{j\sigma}\partial_t^k \psi(t,s,\xi)\|_{L^r}\|\hat u_1\|_{L^{p'}} \\
         & \lesssim\, (1+s)^{\frac{\mu}2}(1+t)^{-\frac{\mu}2}\|u_1\|_{L^p} \\
     &    \times  \begin{cases}
1, \quad r\sigma(k+j-1)+n>0 \\
\ln^{\frac{1}{r}}{(e+s)}, \quad r\sigma(k+j-1)+n = 0\\
(1+s)^{1-k-j-\frac{n}{r\sigma}}, \quad r\sigma(k+j-1)+n<0
\end{cases}\,,
      \end{align*}
thanks to
     \begin{align*} \|  |\xi|^{\sigma(k+j-1)} \|^r_{L^r(Z_1)}&=
     \int_{(1+s)^{-\frac1\sigma}\leq |\xi|\leq 1}|\xi|^{r\sigma(k+j-1)}\ d\xi\\
    &\lesssim \begin{cases}
(1+s)^{-\frac{n}{\sigma}+r(1-k-j)}, \quad r\sigma(k+j-1)+n<0 \\
\ln{(e+s)}, \quad r\sigma(k+j-1)+n = 0\\
1, \quad r\sigma(k+j-1)+n>  0.
\end{cases}
      \end{align*}
%%%%%%%%%%%%%%%%%%%%%%%%%%%%%%%%%%%%%%%%%%%%%%%

\subsection*{Considerations in $Z_{2}$:}
In $Z_2$ we may estimate
\[
 |\chi_2(t, s,\xi) m(t, s, \xi)| \lesssim
\begin{cases}
(1+s)^{-|\rho|}(1+t)^{-\frac12} |\xi|^{\sigma(k+ j-|\rho|-\frac12)}  \, \, \mbox{if} \, \,   \mu\neq 1 \\
(1+t)^{-\frac12} |\xi|^{\sigma(k+
j-\frac12)}\ln\left(\frac{1}{(1+s)|\xi|^\sigma}\right) \, \,
\mbox{if} \, \, \mu=1
\end{cases}
\]
 so that
\[|\chi_2(t, s,\xi)|\xi|^{j\sigma}
\partial_t^k\psi(t,s,\xi)|\lesssim
\begin{cases}
(1+s)^{1-\rho-|\rho|}(1+t)^{\rho-\frac12}|\xi|^{\sigma(k+j-|\rho|-\frac12)} \, \, \mbox{if} \, \,   \mu\neq 1 \\
(1+s)(1+t)^{-\frac12}|\xi|^{\sigma(k+j-\frac12)}
\ln\left(\frac{1}{(1+s)|\xi|^\sigma}\right) \, \, \mbox{if} \, \,
\mu=1.
\end{cases}
\]
If $\mu\neq 1$ and $j + k \leq 1$ then, by using
Haussdorff-Young inequality and H\"{o}lder inequality, setting
\[ \frac1r \doteq \frac1{q'}-\frac1{p'}=\frac1p-\frac1q, \]
for $1\leq p\leq 2\leq q\leq\infty$, one may estimate
\begin{align*}
\|    \mathfrak{F}^{-1}(\chi_2(t, s,\xi)|\xi|^{j\sigma}\partial_t^k
\psi(t,s,\xi))\ast u_1\|_{L^q}
    & \lesssim \|   \chi_2(t, s,\xi)|\xi|^{j\sigma}\partial_t^k   \psi(t,s,\xi)\hat u_1\|_{L^{q'}}\\
     &\lesssim \|  \chi_2(t, s,\xi)|\xi|^{j\sigma}\partial_t^k \psi(t,s,\xi)\|_{L^r}\|\hat u_1\|_{L^{p'}} \\
         & \lesssim\, (1+s)^{1-\rho-|\rho|}(1+t)^{\rho-\frac12}\|u_1\|_{L^p} \\
     &    \times  \begin{cases}
(1+s)^{-\frac{n}{r\sigma}+|\rho|+\frac12-k- j}, \quad r\sigma(k+ j-|\rho|-\frac12)+n>0 \\
\ln^{\frac{1}{r}} {\left ( \frac{e+t}{e+s}\right )}, \quad r\sigma(k+ j-|\rho|-\frac12)+n=0 \\
(1+t)^{-\frac{n}{r\sigma}+|\rho|+\frac12-k- j}, \quad
r\sigma(k+j-|\rho|-\frac12)+n<0
\end{cases}\,,
      \end{align*}
thanks to
     \begin{align*} \|  \chi_2(t, s,\xi)|\xi|^{\sigma(k+j-|\rho|-\frac12)} \|^r_{L^r}&=
     \int_{(1+t)^{-\frac1\sigma} \leq |\xi|\leq(1+s)^{-\frac1\sigma}} |\xi|^{r\sigma(k+j-|\rho|-\frac12)}\ d\xi\\
    &\lesssim \begin{cases}
(1+s)^{-\frac{n}{\sigma}+r(|\rho|+\frac12-k - j)}, \quad r\sigma(k+j-|\rho|-\frac12)+n>0 \\
\ln {\left ( \frac{e+t}{e+s}\right )}, \quad r\sigma(k+j-|\rho|-\frac12)+n=0 \\
(1+t)^{-\frac{n}{\sigma}+r(|\rho|+\frac12-k- j)}, \quad r\sigma(k +
j-|\rho|-\frac12)+n<0.
\end{cases}
      \end{align*}
In particular, if $\mu> \max\left\{1, 2(k + j)+
\frac{2n}{r\sigma}\right\}$ we conclude
\[\|    \mathfrak{F}^{-1}(\chi_2(t, s,\xi)|\xi|^{j\sigma}\partial_t^k \psi(t,s,\xi))\ast u_1\|_{L^q}
\lesssim\, (1+s)(1+t)^{-\frac{n}{r\sigma}-k-j}\|u_1\|_{L^p}.
\]
If $\mu = 1$ and $j+k=1$ we may estimate
\[|\xi|^{\sigma(k+j-\frac12)} \ln\left(\frac{1}{(1+s)|\xi|^\sigma}\right)\lesssim (1+s)^{-\frac12} \]
and  obtain
\[
\|    \mathfrak{F}^{-1}(\chi_2(t, s,\xi)|\xi|^{j\sigma}\partial_t^k
    \psi(t,s,\xi))\ast u_1\|_{L^q}
    \lesssim  (1+s)^{-\frac{n}{r\sigma}+\frac12} (1+t)^{-\frac12}\|u_1 \|_{L^p}
\]
whereas if $\mu = 1$ and $j=k=0$ we may estimate
\[|\xi|^{\sigma(k+j-\frac12)} \ln\left(\frac{1}{(1+s)|\xi|^\sigma}\right) \lesssim |\xi|^{-\frac{\sigma}{2}}\ln\Big(\frac{e+t}{e+s}\Big) \]
and  obtain
\begin{align*}
    \|    \mathfrak{F}^{-1}(\chi_2(t, s,\xi)
    \psi(t,s,\xi))\ast u_1\|_{L^q}
    & \lesssim \|   \chi_2(t, s,\xi)   \psi(t,s,\xi)\hat u_1\|_{L^{q'}}\\
    &\lesssim \|  \chi_2(t, s,\xi) \psi(t,s,\xi)\|_{L^r}\|\hat u_1\|_{L^{p'}} \\
    & \lesssim\, \|u_1\|_{L^p} (1+s)(1+t)^{-\frac12} \ln\Big(\frac{e+t}{e+s}\Big)   \begin{cases}
    (1+s)^{-\frac{n}{r\sigma}+\frac12}, \quad 2n>r\sigma \\
    \ln^{\frac{1}{r}} {\left ( \frac{e+t}{e+s}\right )}, \quad 2n=r\sigma \\
    (1+t)^{-\frac{n}{r\sigma}+\frac12}, \quad
    2n<r\sigma
    \end{cases}\,,
    \end{align*}
    thanks to
    \begin{align*} \|  \chi_2(t, s,\xi)|\xi|^{-\frac \sigma 2} \|^r_{L^r}&=
    \int_{(1+t)^{-\frac1\sigma} \leq |\xi|\leq(1+s)^{-\frac1\sigma}}|\xi|^{-\frac{r\sigma}{2}}\ d\xi \lesssim \begin{cases}
    (1+s)^{-\frac{n}{\sigma}+\frac r2}, \quad 2n>r\sigma \\
    \ln {\left ( \frac{e+t}{e+s}\right )}, \quad 2n=r\sigma \\
    (1+t)^{-\frac{n}{\sigma}+\frac r2}, \quad 2n<r\sigma.
    \end{cases}
    \end{align*}
%%%%%%%%%%%%%%%%%%%%%%%%%%%

\subsection*{Considerations in $Z_{3}$:}

In this zone, since $H^{\pm}_{\rho}=J_{\rho} \pm i Y_{\rho}$ we use
the following representation for the multiplier:
\begin{equation}\label{Bessel-Inteiro}
m(t, s, \xi)=2i|\xi|^{(k+j)\sigma} \left|
                         \begin{array}{cc}
                           J_{\rho}\left( (1+s)|\xi|^{\sigma} \right) & J_{\rho-k}\left( (1+t)|\xi|^{\sigma} \right) \\
                           Y_{\rho}\left( (1+s)|\xi|^{\sigma} \right) &  Y_{\rho-k}\left( (1+t)|\xi|^{\sigma} \right) \\
                         \end{array}
                       \right|
\end{equation}
if $\rho, \rho-k\in \mathbf{Z}$, or
\begin{equation}\label{Bessel}
m(t, s, \xi)=2i\csc(\rho \pi)|\xi|^{(k+j)\sigma} \left|
                         \begin{array}{cc}
                           J_{-\rho}\left( (1+s)|\xi|^{\sigma} \right) & J_{-\rho+k}\left( (1+t)|\xi|^{\sigma} \right) \\
                           (-1)^{k}J_{\rho}\left( (1+s)|\xi|^{\sigma} \right) &  J_{\rho-k}\left( (1+t)|\xi|^{\sigma} \right) \\
                         \end{array}
                       \right|
\end{equation}
if $\rho, \rho-k\not\in \mathbf{Z}$, with $k=0,1$ and $J_{\rho},
Y_{\rho}$ denote the Bessel functions of the first and second kind,
respectively. We apply Lemma \ref{hankel} (see
Appendix) in the following estimates to both cases, which are
slightly different. In the case $\rho, \rho-k\not\in \mathbf{Z}$ we
obtain
$$
 |\chi_3(t, s,\xi) m(t, s, \xi)| \lesssim
(1+s)^{-\rho}(1+t)^{\rho-k}|\xi|^{j\sigma} +
(1+s)^{\rho}(1+t)^{-\rho+k}|\xi|^{(2k+j)\sigma},
$$
so that
\[|\chi_3(t, s,\xi)|\xi|^{j\sigma}
\partial_t^k\psi(t,s,\xi)|\lesssim
(1+s)^{1-2\rho}(1+t)^{2\rho-k}|\xi|^{j\sigma} +
(1+s)(1+t)^{k}|\xi|^{(2k+j)\sigma}.
\]
By using Haussdorff-Young inequality and H\"older inequality,
setting
\[ \frac1r = \frac1{q'}-\frac1{p'}=\frac1p-\frac1q, \]
for $1\leq p\leq 2\leq q\leq\infty$ one may estimate for $ k+j=0,1$,
\begin{align*}
\|   \mathfrak{F}^{-1}(\chi_3(t, s,\xi)|\xi|^{j\sigma}\partial_t^k
\psi(t,s,\xi))\ast u_1\|_{L^q}
    & \lesssim \| |\xi|^{j\sigma} \chi_3(t, s,\xi)\partial_t^k   \psi(t,s,\xi)\hat u_1\|_{L^{q'}}\\
     &\lesssim \|   \chi_3(t, s,\xi)|\xi|^{j\sigma}\partial_t^k   \psi(t,s,\xi)\|_{L^r} \|\hat
     u_1\|_{L^{p'}} \\
&\lesssim\,(1+t)^{-\frac{n}{\sigma}\left(\frac1p-\frac1q\right)+\rho+|\rho|-k-j}(1+s)^{1-\rho-|\rho|}
\|u_1\|_{L^p}\,
\end{align*}
thanks to
$$
 \|  \chi_3(t, s,\xi)|\xi|^{a\sigma} \|^r_{L^r}=
     \int_{|\xi|\leq(1+t)^{-\frac1\sigma}}\, |\xi|^{ra\sigma}
     d\xi \lesssim (1+t)^{-\frac{n}{\sigma}-ar},
$$
with $a\geq 0$.
In the case, $\rho, \rho-k\in \mathbf{Z}$, we
obtain
\begin{eqnarray*}
|\chi_3(t, s,\xi) m(t, s, \xi)|   & \lesssim &
(1+s)^{-\rho}(1+t)^{\rho-k}|\xi|^{j\sigma} +
(1+s)^{\rho}(1+t)^{-\rho+k}|\xi|^{(2k+j)\sigma} \\
   &+ & (1+s)^{\rho}(1+t)^{-\rho+k}|\xi|^{(2k+j)\sigma}\ln{\left ( \frac{e+t}{e+s} \right )}
\end{eqnarray*}
if $\rho-k \geq 0$ or
\begin{eqnarray*}
|\chi_3(t, s,\xi) m(t, s, \xi)|   & \lesssim &
(1+s)^{-\rho}(1+t)^{\rho-k}|\xi|^{j\sigma} +
(1+s)^{\rho}(1+t)^{-\rho+k}|\xi|^{(2k+j)\sigma} \\
   &+ & (1+s)^{-\rho}(1+t)^{\rho-k}|\xi|^{j\sigma}\ln{\left ( \frac{e+t}{e+s} \right )}
\end{eqnarray*}
if $\rho-k < 0$. In fact, we
use the relation $J_{\rho-k}((1+t)|\xi|^{\sigma}) =(-1)^{k-\rho}
J_{k-\rho}((1+t)|\xi|^{\sigma})$ if $\rho -k\geq 0$ and
$J_{\rho}((1+s)|\xi|^{\sigma}) =(-1)^{-\rho}
J_{-\rho}((1+s)|\xi|^{\sigma})$ if $\rho -k< 0$. By using
Haussdorff-Young inequality and H\"older inequality, setting
\[ \frac1r = \frac1{q'}-\frac1{p'}=\frac1p-\frac1q, \]
for $1\leq p\leq 2\leq q\leq\infty$ one may estimate for $ k+j=0,1$,
\begin{align*}
\|   \mathfrak{F}^{-1}(\chi_3(t, s,\xi)|\xi|^{j\sigma}\partial_t^k
\psi(t,s,\xi))\ast u_1\|_{L^q}
    & \lesssim \| |\xi|^{j\sigma} \chi_3(t, s,\xi)\partial_t^k   \psi(t,s,\xi)\hat u_1\|_{L^{q'}}\\
     &\lesssim \|   \chi_3(t, s,\xi)|\xi|^{j\sigma}\partial_t^k   \psi(t,s,\xi)\|_{L^r} \|\hat
     u_1\|_{L^{p'}} \\
&\lesssim\, \|u_1\|_{L^p}  \left \{ \begin{array}{ccc}
(1+t)^{\rho+|\rho|-k-j-\frac{n}{\sigma}\left(\frac1p-\frac1q\right)} (1+s)^{1-\rho-|\rho|}  & \mbox{if} &\rho \neq 0 \\
(1+s)(1+t)^{-k-j-\frac{n}{\sigma}\left(\frac1p-\frac1q\right)} \left(
1+\ln{\left ( \frac{e+t}{e+s} \right )} \right) & \mbox{if}
& \rho = 0.
 \end{array} \right.
\end{align*}
Hence, if $\mu \neq 1$, then we have the same
estimate for both cases.

In particular, if $\mu> 1$ we conclude
\[ \|   \mathfrak{F}^{-1}(\chi_3(t,
s,\xi)|\xi|^{j\sigma}\partial_t^k \psi(t,s,\xi))\ast u_1\|_{L^q}
\lesssim\,
(1+s)(1+t)^{-\frac{n}{\sigma}\left(\frac1p-\frac1q\right)-k-j}\|u_1\|_{L^p},
\]
and, if $\mu< 1$ we conclude
\[ \|   \mathfrak{F}^{-1}(\chi_3(t,
s,\xi)|\xi|^{j\sigma}\partial_t^k \psi(t,s,\xi))\ast u_1\|_{L^q}
\lesssim\,
(1+s)^{\mu}(1+t)^{1-\mu-\frac{n}{\sigma}\left(\frac1p-\frac1q\right)-k-j}\|u_1\|_{L^p}.
\]
%\end{itemize}
\subsection*{Considerations in $Z_{high}$:}

Thanks to Lemma \ref{hankel} (see
Appendix), we may decompose $m_0\doteq (1-\chi)m$ as the sum of two
multipliers
\[ e^{\pm i(t-s)|\xi|^\sigma}|\xi|^{(k+j)\sigma} a((1+s)|\xi|^\sigma) b((1+t)|\xi|^\sigma),\]
where $a,b$ are symbols of order $-\frac12$.

If one try to follow the analysis of the previous zones, in $Z_{high}$ it appears the additional restriction $\frac{n}{\sigma}\left(\frac1{p}-\frac1{q}\right)<1$ on the $L^p-L^q$ estimates, for $1\leq p\leq 2\leq q\leq\infty$. To relax this range, in this zone of the extended phase space we may employ the strategy used in \cite{EL} to study the damping-free
problem.
By using duality argument, it is enough to prove the estimates for $\frac1p+\frac1q\geq 1$.\\

Let $\phi\in C^{\infty}_c\left(\R^n\right)$ be a non-negative function supported  in $\{\xi: \frac{1}{2}\leq |\xi| \leq 2\}$ and $\phi_\ell(\xi)\doteq \phi(2^{-\ell}|\xi|)$,
with $\ell$  an integer satisfying
\[ \sum_{\ell\in \Z} \phi_\ell(\xi)=1, \qquad \forall \xi\neq 0.\]
In particular, $(1-\chi)\phi_\ell =0$ if $\ell<-1$ and
$(1-\chi)\phi_\ell =\phi_\ell$ if $\ell\geq 1$,
hence one may write
\[\phi_\ell(\xi)m_0(t,s, \xi)=\sum_{\ell=-1}^{\infty} \phi_\ell(\xi)m_0(t,s, \xi).\]
By using   Plancherel's theorem and  putting $\eta\doteq 2^{-\ell}\xi$ we have
\begin{equation}\label{11}
 \|\phi_\ell \cdot (1-\chi)m(t,s, \cdot)\|_{M_2^2}=\sup_{\eta\in supp \phi} |\phi(\eta)m_0(t, x, 2^\ell|\eta|)| \ \leq
 C2^{\ell(k+j-1)\sigma}(1+t)^{-\frac12}(1+s)^{-\frac12}.
\end{equation}
Now, by using  Littman's lemma (see Appendix) we
conclude
\begin{align*}
&\Big\|\mathfrak{F}^{-1}_{\xi\rightarrow x}\Big(e^{\pm i(t-s)|\xi|^{\sigma}}\phi_\ell(\xi) |\xi|^{(k+j)\sigma}a((1+s)|\xi|^\sigma) b((1+t)|\xi|^\sigma)\Big)\Big\|_{L^{\infty}}\\
&=2^{\ell( n + (k+j)\sigma)}\Big\|\mathfrak{F}^{-1}_{\eta\rightarrow x}\Big(e^{\pm i(t-s)2^{\ell\sigma}|\eta|^{\sigma}}\phi(\eta)|\eta|^{(k+j)\sigma}a((1+s)2^{\ell\sigma}|\eta|^\sigma) b((1+t)2^{\ell\sigma}|\eta|^\sigma)\Big\|_{L^{\infty}}  \\
&\leq C 2^{\ell(n+(k+j)\sigma}) (1+(t-s)2^{\ell\sigma})^{-\frac{n}{2}}\sum_{|\alpha|\leq L} \|D_{\eta}^\alpha \phi(\eta)|\eta|^{(k+j)\sigma}a((1+s)2^{\ell\sigma}|\eta|^\sigma) b((1+t)2^{\ell\sigma}|\eta|^\sigma) \|_{L^\infty}\\
&\leq C
2^{\ell(n+(k+j-1)\sigma)}(1+(t-s)2^{\ell\sigma})^{-\frac{n}{2}}(1+t)^{-\frac12}(1+s)^{-\frac12}.
\end{align*}
 We remark that for $\sigma\neq 1$ the rank of the Hessian $H_{|\eta|^\sigma}$ is equal to $n$.\\
 Hence,  Young's Inequality implies
\begin{align}\label{des11}
\Big\|\mathfrak{F}^{-1}_{\xi\rightarrow x}\Big(
m_{0}(t,s,\cdot)\phi_\ell(\xi)\mathfrak{F}(f)\Big)\Big\|_{L^{\infty}(\R^n)}\leq
C 2^{\ell(n+(k+j-1)\sigma-\frac{n}{2}\sigma)}
(t-s)^{-\frac{n}{2}}(1+t)^{-\frac12}(1+s)^{-\frac12}\|f\|_{L^1},
\end{align}
for all integer $\ell$, or equivalent,
\begin{equation}\label{1}
\|\phi_\ell \cdot m_{0}(t,s,\cdot)\|_{M_1^{\infty}}\leq
C2^{\ell(n+(k+j-1)\sigma-\frac{n}{2}\sigma)}
(t-s)^{-\frac{n}{2}}(1+t)^{-\frac12}(1+s)^{-\frac12}.
\end{equation}
%\textcolor{red}{Here, as in Wirth, we can have not singular estimate by putting away the term $2^{-\frac{kn}{2}\sigma}$, but for now I prefer to enlarge the zone where we have $L^p-L^q$ estimates. Moreover, singular estimates can be used for the semilinear problem!}
As a consequence of \eqref{11},  \eqref{1} and the Riesz-Thorin interpolation theorem
we get
\begin{equation}\label{2}
\|\phi_\ell \cdot  m_{0}(t,s,\cdot)\|_{M_{p_0}^{q_0}}\leq
C2^{\ell\left((k+j-1)\sigma+n\left(\frac1{p_0}-\frac1{q_0}\right)\left(1-\frac{\sigma}{2}\right)\right)}
(t-s)^{-\frac{n}{2}\left(\frac1{p_0}-\frac1{q_0}\right)}(1+t)^{-\frac12}(1+s)^{-\frac12}
\end{equation}
for $\frac{1}{p_0}+\frac{1}{q_0}=1$.

In order to derive an estimate for $\|\phi_\ell \cdot m_{0}\|_{M_1^1}$,  one may prove the following estimates
\begin{eqnarray*} \|\partial_\xi^\gamma (\phi_\ell  m_{0}(t,s,\cdot))\|_{L^2}&\leq & C (t-s)^{|\gamma|}
(1+t)^{-\frac12}(1+s)^{-\frac12}\Big(\int_{2^{\ell-1}\leq|\xi|\leq 2^{\ell+1}} |\xi|^{2(k+j)\sigma-2\sigma+2(\sigma-1)|\gamma|}\,d\xi \Big)^{\frac12}\\
& \leq &
C_1\,2^{\ell(\frac{n}{2}+(k+j-1)\sigma+|\gamma|(\sigma-1))}(t-s)^{|\gamma|}(1+t)^{-\frac12}(1+s)^{-\frac12}
\end{eqnarray*}
and applying  the Berstein's inequality (see
Proposition \ref{sigrid1} in Appendix) for $N>\frac{n}{2}$ we get
\begin{eqnarray}\label{4}
\|\phi_\ell  m_{0}(t,s,\cdot)\|_{M_1^1}&\leq& \|\phi_\ell m_{0}(t,s,\cdot)\|_{L^2}^{1-\frac{n}{2N}}\|D^N \left(\phi_\ell m_{0}(t,s,\cdot) \right)\|_{L^2}^{\frac{n}{2N}}\nonumber\\
 &\leq& C 2^{\ell\sigma(\frac{n}{2}+ k+j-1)}(t-s)^{\frac{n}{2}}(1+t)^{-\frac12}(1+s)^{-\frac12}.
\end{eqnarray}
Using \eqref{2},  \eqref{4} and  Riesz-Thorin interpolation theorem we conclude that
\begin{align*}
 \|\phi_\ell m_{0}(t,s,\cdot)\|_{M_{p}^{q}}\leq  C 2^{\ell n\left(\frac{1}{p}+\frac{\sigma-1}{q}-\sigma\big(\frac{1}{2}-\frac{k+j-1}{n}\big)\right)}
 (t-s)^{\frac{n}{2}\left(\theta-\left(\frac{1}{p}-\frac{1}{q}\right)\right) }(1+t)^{-\frac12}(1+s)^{-\frac12},
\end{align*}
where $0< \theta < 1$, with $\frac{1}{p}=\frac{1-\theta}{p_0}+\theta$ and $\frac{1}{q}=\frac{1-\theta}{q_{0}}+\theta$.

Therefore, for large frequencies, using the
Littlewood-Paley dyadic decomposition we conclude
the estimate
\begin{equation}\label{eq:partition}
  \|m_{0}(t,s,\cdot)\|_{M_{p}^{q}} \leq  \sum_{\ell \geq -1}\|\phi_\ell m_{0}(t,s,\cdot)\|_{M_{p}^{q}}\leq C (t-s)^{\frac{n}{2}\left(-1+\frac{2}{q}\right) }(1+t)^{-\frac12}(1+s)^{-\frac12},
\end{equation}
which is convergent if
\begin{equation}\label{sharp}
\frac{1}{p}+\frac{\sigma-1}{q}<\sigma\big(\frac{1}{2}-\frac{k+j-1}{n}\big),
\qquad \mbox{for} \ \frac1p+\frac1q\geq 1.
\end{equation}
 By duality arguments,  the analogous estimate is true if
%\begin{equation}\label{sharp1}
\[
\frac{1-\sigma}{p}-\frac{1}{q}<\sigma\big(-\frac{k+j-1}{n}-\frac{1}{2}\big),
\qquad \mbox{for} \ \frac1p+\frac1q\leq 1.
\]
%\end{equation}
However, in the special case~$1<p\leq 2\leq q<\infty$, the latter
estimates may be refined by using the embeddings for Besov spaces
(see, for instance, \cite{ST}): $L^p\hookrightarrow B^0_{p,2}$
for~$p\in(1,2]$ and $B^0_{q,2}\hookrightarrow L^q$
for~$q\in[2,\infty)$. Indeed, since the sum in~\eqref{eq:partition}
is finite for any given~$\xi$ (see Appendix), in
particular, $\#\{\ell: \ \phi_{\ell}(\xi)\neq0\}\leq 3$, we obtain
the chain of inequalities (see
also~\cite{Brenner1})
\[ \|\mathfrak{F}^{-1}(m_0\hat f)\|_{B^0_{q,2}}\leq C_1\sup_{\ell} \|\mathfrak{F}^{-1}(m_0\phi_{\ell}\hat f)\|_{L^q} \leq C_2\,\|f\|_{L^p}\leq C_3 \,\|f\|_{B^0_{p,2}}. \]

Summing up we have:
\begin{proposition} Let $n\in \N$ and $\sigma\neq 1$.
Assume  $1\leq p \leq q\leq \infty$ and $j+k\leq 1$ such that
\begin{equation}\label{eq:rangeboth}
\frac{n}\sigma\left(\frac1p-\frac1q\right) + n\max\left\{\left(\frac12-\frac1p\right),\left(\frac1q-\frac12\right)\right\} +j+k<1.
\end{equation}
Then  $u^{high}(t,s,\cdot)\doteq \mathfrak{F}^{-1}((1-\chi(|\xi|))
\psi(t,s,\xi))\ast u_1(x)  $ satisfies
\begin{equation}\label{linearestimates}
\|\partial_t^k(-
\Delta)^{\frac{j\sigma}{2}}u^{high}(t,s,\cdot)\|_{L^q} \leq C
(t-s)^{n\Gamma(p,q)}(1+t)^{-\frac{\mu}{2}}(1+s)^{\frac{\mu}{2}}\|u_1\|_{L^p},
\end{equation}
with
\[
\Gamma(p,q)=
\begin{cases}
\frac{1}{q}-\frac{1}{2}, \quad \frac1p+\frac1q\geq 1\\
\frac{1}{2}-\frac{1}{p}, \quad \frac1p+\frac1q\leq 1.
\end{cases}
\]
Moreover,  if equality holds in \eqref{eq:rangeboth},  estimate
\eqref{linearestimates} remains valid for  $1<p\leq 2\leq q<\infty$.
\end{proposition}
%\begin{remark} If we put $p=1$ in \eqref{sharp},  \eqref{linearestimates} is true for all $ \tilde{q}< q \leq\infty$($q\geq \tilde{q}\geq2$), with $\tilde{q}$ given by
%\begin{equation}\label{fixed}
%\frac{1}{\tilde{q}}\doteq\frac{1}{\sigma-1} \left( \sigma\left( \frac1{n} + \frac1{2}\right) -1 \right).
%\end{equation}
%For $2\sigma=n$ we get $\tilde{q}=2$, whereas $\tilde{q}<2$ for $2\sigma>n$ and $\tilde{q}>2$ for $2\sigma<n$.
%\end{remark}
\begin{remark}
If $\frac1p+\frac1q\geq 1$ with $1\leq p \leq 2$, \eqref{sharp} is
true  for all $ \tilde{q}< q \leq\infty$ ($q\geq \tilde{q}\geq2$),
with $\tilde{q}$ given by
\begin{equation}\label{fixed}
\frac{1}{\tilde{q}}\doteq\frac{1}{\sigma-1} \left( \sigma\left(
\frac{1-j-k}{n} + \frac1{2}\right) - \frac{1}{p} \right).
\end{equation}
In particular, if  $j=k=0$, then for $2\sigma=\left ( \frac{2}{p} -
1\right )n$ we get $\tilde{q}=2$, whereas $\tilde{q}<2$ for
$2\sigma> \left ( \frac{2}{p} - 1\right )n$ and $\tilde{q}>2$ for
$2\sigma< \left ( \frac{2}{p} - 1\right )n$. If $j+k=1$, then
$p=2=\tilde{q}$ and  $1\leq p < 2<\tilde{q}$.
\end{remark}

\begin{remark} For $p=1$ the term   $(t-s)^{n\Gamma(p,q)}$ may be singular but $n\Gamma(p,q)>-1$ for  $q<\frac{2n}{[n-2]_+}$.
\end{remark}
\begin{remark}
For $n <2\sigma(1-k-j)$ we may also have $L^1-L^2$ estimate for
$u^{high}$. Indeed,
\begin{eqnarray*}
\|\partial_t^k(-
\Delta)^{\frac{j\sigma}{2}} u^{high}(t,s,\cdot)\|_{L^2}&\lesssim& \| (1-\chi(|\xi|))|\xi|^{j\sigma}\partial_t^k \psi(t,s,\xi)\|_{L^2}\|\hat u_1\|_{L^{\infty}} \\
         & \lesssim& \||\xi|^{(k+j-1)\sigma}\|_{L^2(|\xi|\geq 1)}(1+s)^{\frac{\mu}2}(1+t)^{-\frac{\mu}2}\|u_1\|_{L^1}\\& \lesssim & (1+s)^{\frac{\mu}2}(1+t)^{-\frac{\mu}2}\|u_1\|_{L^1}.
\end{eqnarray*}
\end{remark}

\begin{corollary}\label{corHigh}
Let $u_1\in L^2$, then  $u^{high}(t,s,\cdot)\doteq \mathfrak{F}^{-1}((1-\chi(|\xi|)) \psi(t,s,\xi))\ast u_1(x)  $ satisfies
\begin{equation}\label{linearestimatesb}
\|u^{high}(t,s,\cdot)\|_{L^q} \leq C
(1+t)^{-\frac{\mu}{2}}(1+s)^{\frac{\mu}{2}}\|u_1\|_{L^2}, \quad
2\leq q\leq \frac{2n}{[n-2\sigma]_+}.
\end{equation}
and for $j+k=1$
\begin{equation}\label{linearestimatesderivatives}
\|\partial_t^k(- \Delta)^{\frac{j\sigma}{2}}u^{high}(t,s,\cdot)\|_{L^2} \leq C (1+t)^{-\frac{\mu}{2}}(1+s)^{\frac{\mu}{2}}\|u_1\|_{L^2}.
\end{equation}
\end{corollary}
%%%%%%%%%%%%%%%%%%%%%%%%%%%%%%%%%%%%%%%%%%%%%%%%%%%%%
\begin{remark}
It is worth to mention that different from $Z_{high}$ and $Z_1$, in zones $Z_2$ and $Z_3$ additional derivatives produce additional decay.
\end{remark}

In the following, we state the  estimates for solutions to the linear problem for $s=0$ and $s\neq 0$ that will be used in Section \ref{NL}.

\begin{theorem}\label{theoremlinear}
Let $\sigma > 1$, $1 \leq n < 2 \sigma$ and $q\geq 2$.
\begin{description}
 \item[(i)] If $\mu> \max \{ \frac{2n}{\sigma} \left ( 1-\frac{1}{q} \right
  ); 1 \}$,  then the solution to \eqref{mainlineq} satisfies
  \[\|   u(t,\cdot)\|_{L^q}
\lesssim\,(1+t)^{-\frac{n}{\sigma}\left(1-\frac1{q}\right)}(1+s)
\left(  \|u_1\|_{L^1} + (1+s)^{ \frac{n}{2\sigma}}
\|u_1\|_{L^2}\right)
\]
or
\[\|   u(t,\cdot)\|_{L^q}
\lesssim\,(1+t)^{-\frac{n}{\sigma}\left(1-\frac1{q}\right)} \left( (1+s)^{\max\left\{1 , \frac{n}{\sigma}\left(1-\frac1{q}\right)\right\}} \|u_1\|_{L^1} + (1+s)^{ \frac{n}{\sigma}\left(1-\frac1{q}\right)} \|u_1\|_{L^2}\right), \quad \frac{n}{\sigma}\left(1-\frac1{q}\right)\neq 1;
\]
\item[(ii)]
If $\max \{  2-\frac{2n}{\sigma} \left ( 1-\frac{1}{q} \right
  ) ; 0  \} < \mu < 1$ or $1 < \mu < \frac{2n}{\sigma} \left ( 1-\frac{1}{q} \right ) $, then
the solution to \eqref{mainlineq} satisfies
\[\| u(t,\cdot)\|_{L^q}
\lesssim \, (1+t)^{-\frac{\mu}{2}} (1+s)^{
1+\frac{\mu}{2}-\frac{n}{\sigma}\left(1-\frac1{q}\right) } \left(
\|u_1\|_{L^1} +(1+s)^{\frac{n}{2\sigma}}\|u_1\|_{L^2} \right)
\]
or
\[\| u(t,\cdot)\|_{L^q}
\lesssim \, (1+t)^{-\frac{\mu}{2}} (1+s)^{
\frac{\mu}{2}} \left(
(1+s)^{\max\left\{0 , 1-\frac{n}{\sigma}\left(1-\frac1{q}\right)\right\}}\|u_1\|_{L^1} +\|u_1\|_{L^2} \right), \quad \frac{n}{\sigma}\left(1-\frac1{q}\right)\neq 1.
\]
Moreover, if $\mu=\frac{2n}{\sigma} \left (1-\frac{1}{q} \right )
>1$, then
 \[\| u(t,\cdot)\|_{L^q}
\lesssim \,
(1+t)^{-\frac{n}{\sigma}\left(1-\frac1{q}\right)}(1+s)\left(
\left(\ln {\left ( \frac{e+t}{e+s}\right )}\right)^{1-\frac{1}{q}}
 \|u_1\|_{L^1}
+(1+s)^{\frac{n}{2\sigma}} \|u_1\|_{L^2} \right)
\]
whereas if $ \mu = 2-\frac{2n}{\sigma} \left ( 1-\frac{1}{q} \right
  ) <1$, then
 \[\| u(t,\cdot)\|_{L^q}
\lesssim \,
(1+t)^{\frac{n}{\sigma}\left(1-\frac1{q}\right)-1}(1+s)^{\mu}\left(
\left(\ln {\left ( \frac{e+t}{e+s}\right )}\right)^{1-\frac{1}{q}}
 \|u_1\|_{L^1}
+(1+s)^{\frac{n}{2\sigma}} \|u_1\|_{L^2} \right);
\]

\item[(iii)]
 If $ 0 < \mu < \min \{ 2-\frac{2n}{\sigma} \left
( 1-\frac{1}{q} \right
  ) ; 1 \}, $
then the solution to \eqref{mainlineq} satisfies
\[\| u(t,\cdot)\|_{L^q}
\lesssim \,
(1+t)^{1-\mu-\frac{n}{\sigma}\left(1-\frac1{q}\right)}(1+s)^{\mu}
\left(  \|u_1\|_{L^1} +(1+s)^{\frac{n}{2\sigma}}\|u_1\|_{L^2}
\right)
\]
or
\[\| u(t,\cdot)\|_{L^q}
\lesssim \, (1+t)^{1-\mu-\frac{n}{\sigma}\left(1-\frac1{q}\right)}
\left( (1+s)^{\max\left\{\mu ,
\mu-1+\frac{n}{\sigma}\left(1-\frac1{q}\right)
\right\}}\|u_1\|_{L^1} +
(1+s)^{\mu-1+\frac{n}{\sigma}\left(1-\frac1{q}\right)}\|u_1\|_{L^2}
\right), \quad \frac{n}{\sigma}\left(1-\frac1{q}\right)\neq 1.
\]

\item[(iv)]  If $\mu=1,$  then the solution to
\eqref{mainlineq} satisfies
\begin{eqnarray*}
 \| u(t, \cdot)\|_{L^q}  & \lesssim & (1+t)^{-\frac{1}{2}}(1+s)^{\frac{3}{2}-\frac{n}{\sigma}\left(
\frac{1}{2}-\frac1q\right)} \|u_1\|_{L^2} \\
   &+ & \|u_1\|_{L^1} (1+t)^{-\min \{ \frac{n}{\sigma}\left(
1-\frac1q\right); \frac{1}{2} \} }
\begin{cases}
(1+s)^{\frac{3}{2}-\frac{n}{\sigma}\left(
1-\frac1q\right)}\ln{\left( \frac{e+t}{e+s} \right)} , \quad  q > \frac{2n}{[2n-\sigma]_+}
\\
(1+s)\ln^{2-\frac{1}{q}}{\left( \frac{e+t}{e+s} \right)} , \quad q =  \frac{2n}{[2n-\sigma]_+}
\\
(1+s)\ln{\left( \frac{e+t}{e+s} \right)} , \quad 1 \leq  q <  \frac{2n}{[2n-\sigma]_+} .
\end{cases}
\end{eqnarray*}

\end{description}
\end{theorem}
\begin{remark}
We point out that
$$
 \max \left \{ \frac{2n}{\sigma} \left (
1-\frac{1}{q} \right
  ); 1 \right \}  =\left \{
                      \begin{array}{ccc}
                        1 & if &  1 \leq  q \leq  \frac{2n}{[2n-\sigma]_+}\\
                        \frac{2n}{\sigma} \left ( 1-\frac{1}{q} \right
  ) & if &   q >  \frac{2n}{[2n-\sigma]_+} \\
                      \end{array}
                    \right.
$$
is equivalent to
$$
\min \left \{ 2-\frac{2n}{\sigma} \left ( 1-\frac{1}{q} \right
  ) ; 1 \right\} = \left \{
                      \begin{array}{ccc}
                        1 & if &  1 \leq q \leq  \frac{2n}{[2n-\sigma]_+} \\
                        2-\frac{2n}{\sigma} \left ( 1-\frac{1}{q} \right
  )  & if &    q >  \frac{2n}{[2n-\sigma]_+}\\
                      \end{array}
\right.\,.
$$
\end{remark}
\begin{proof}
For $n < 2 \sigma$, from \eqref{linearestimatesb} we get
\[ \|u^{high}(t,s,\cdot)\|_{L^q} \leq C
(1+t)^{-\frac{\mu}{2}}(1+s)^{\frac{\mu}{2}}\|u_1\|_{L^2}, \quad q
\geq 2.\]
The proof of (i): Suppose that $\mu> \max \{
\frac{2n}{\sigma} \left ( 1-\frac{1}{q} \right
  ); 1 \}$. If $\mu> \frac{2n}{\sigma} \left ( 1-\frac{1}{q}
  \right)$ then
\[(1+t)^{-\frac{\mu}{2}}(1+s)^{\frac{\mu}{2}}\leq (1+t)^{-\frac{n}{\sigma}\left(1-\frac1{q}\right)}(1+s)^{\frac{n}{\sigma}\left(1-\frac1{q}\right)}. \]
For $n < 2 \sigma$ we get
\[ \|u^{high}(t,s,\cdot)\|_{L^q} \leq C
(1+t)^{-\frac{n}{\sigma}\left(1-\frac1{q}\right)}(1+s)^{\frac{n}{\sigma}\left(1-\frac1{q}\right)}\|u_1\|_{L^2},
\quad q \geq 2.
\]
Applying the derived estimates at zone $Z_1$ with $p=2$ or $p=1$, respectively,  we may estimate
\begin{eqnarray*}
\|   \mathfrak{F}^{-1}(\chi(|\xi|)\chi_1(s,\xi) \psi(t,s,\xi))\ast u_1\|_{L^q}
            & \lesssim&\, (1+t)^{-\frac{\mu}2}(1+s)^{\frac{\mu}2+1-\frac{n}{\sigma}\left(\frac12-\frac1{q}\right)}\|u_1\|_{L^2}\\
             & \lesssim& (1+t)^{-\frac{n}{\sigma}\left(1-\frac1{q}\right)}(1+s)^{\frac{n}{\sigma}\left(1-\frac1{q}\right)+1-\frac{n}{\sigma}\left(\frac12-\frac1{q}\right) }\|u_1\|_{L^2}\\
             &=& (1+t)^{-\frac{n}{\sigma}\left(1-\frac1{q}\right)}(1+s)^{1+\frac{n}{2\sigma}}\|u_1\|_{L^2}\end{eqnarray*}
             or
\begin{eqnarray*}
\|   \mathfrak{F}^{-1}(\chi(|\xi|)\chi_1(s,\xi) \psi(t,s,\xi))\ast u_1\|_{L^q}
            & \lesssim&\, (1+t)^{-\frac{\mu}2}(1+s)^{\frac{\mu}2+ \max\left\{0, 1-\frac{n}{\sigma}\left(1-\frac1{q}\right)\right\}}\|u_1\|_{L^1}\\
             & \lesssim& (1+t)^{-\frac{n}{\sigma}\left(1-\frac1{q}\right)}(1+s)^{\frac{n}{\sigma}\left(1-\frac1{q}\right)+\max\left\{0, 1-\frac{n}{\sigma}\left(1-\frac1{q}\right)\right\} }\|u_1\|_{L^1}\\
             &=& (1+t)^{-\frac{n}{\sigma}\left(1-\frac1{q}\right)}(1+s)^{\max\left\{1 , \frac{n}{\sigma}\left(1-\frac1{q}\right)\right\}}\|u_1\|_{L^1}.\end{eqnarray*}

In $Z_2$ and $Z_3$ we have the following estimate:
\[\|   u(t,\cdot)\|_{L^q}
\lesssim\,(1+t)^{-\frac{n}{\sigma}\left(1-\frac1{q}\right)}  (1+s)
\|u_1\|_{L^1}.\]
Thus (i) is concluded.\\
The proof of (ii): In $Z_1$ we have
\begin{eqnarray*}
  \| u(t,\cdot)\|_{L^q} &  \lesssim &  (1+t)^{-\frac{\mu}2}(1+s)^{1+\frac{\mu}2-\frac{n}{\sigma}\left(\frac12-\frac1{q}\right)}
          \|u_1\|_{L^2} \\
   & = & (1+t)^{-\frac{\mu}{2}} (1+s)^{
1+\frac{\mu}{2}-\frac{n}{\sigma}\left(1-\frac1{q}\right)+
\frac{n}{2\sigma}} \|u_1\|_{L^2}
\end{eqnarray*}
or
\begin{eqnarray*}
\|   \mathfrak{F}^{-1}(\chi(|\xi|)\chi_1(s,\xi) \psi(t,s,\xi))\ast u_1\|_{L^q}
            & \lesssim&\, (1+t)^{-\frac{\mu}2}(1+s)^{\frac{\mu}2+ \max\left\{0, 1-\frac{n}{\sigma}\left(1-\frac1{q}\right)\right\}}\|u_1\|_{L^1}.
             \end{eqnarray*}
Suppose that $1 < \mu \leq
\frac{2n}{\sigma} \left ( 1-\frac{1}{q} \right)$. We have in $Z_2$
\[\|   u(t,\cdot)\|_{L^q}
\lesssim\,(1+t)^{-\frac{\mu}{2}}
(1+s)^{1+\frac{\mu}{2}-\frac{n}{\sigma}\left ( 1-\frac{1}{q}
\right)} \|u_1\|_{L^1}
\begin{cases} \left(\ln {\left (
\frac{e+t}{e+s}\right )}\right)^{1-\frac{1}{q}}, \quad  \mu=\frac{2n}{\sigma} \left (
1-\frac{1}{q} \right ) \\
1, \quad  1<\mu< \frac{2n}{\sigma} \left ( 1-\frac{1}{q} \right
  ).
\end{cases} \,
\]
For  $\mu \leq \frac{2n}{\sigma} \left ( 1-\frac{1}{q} \right)$ we
have
\[
(1+t)^{-\frac{n}{\sigma}\left(1-\frac1{q}\right)}(1+s)\leq
(1+t)^{-\frac{\mu}{2}+\frac{\mu}{2}-\frac{n}{\sigma}\left(1-\frac1{q}\right)}(1+s)\leq
(1+t)^{-\frac{\mu}{2}} (1+s)^{1+\frac{\mu}{2}-\frac{n}{\sigma}\left(1-\frac1{q}\right)},
\]
hence in $Z_3$ we obtain
\begin{eqnarray*}
\| u(t,\cdot)\|_{L^q}  & \lesssim & (1+t)^{-\frac{n}{\sigma}\left(1-\frac1{q}\right)}(1+s) \|u_1\|_{L^1}\\
   & \lesssim &
   (1+t)^{-\frac{\mu}{2}}(1+s)^{1+\frac{\mu}{2}-\frac{n}{\sigma}\left(1-\frac1{q}\right)} \|u_1\|_{L^1}.
\end{eqnarray*}
%
%because of $1-\frac{n}{\sigma}\left(\frac12-\frac1{q}\right)=1-\frac{n}{\sigma}\left(1-\frac1{q}\right)+
%\frac{n}{2\sigma}>0$ for $n<2\sigma$.
%\[\| u(t,\cdot)\|_{L^q}
%\lesssim \, (1+t)^{-\frac{\mu}{2}}
%(1+s)^{\max{\left\{1,\frac{\mu}{2}\right\}}\|u_1\|_{L^1\cap L^2}};
%\]
%for $ 1\leq n \leq \frac{q}{q-1}\sigma $
%\[\| u(t,\cdot)\|_{L^q}
%\lesssim \, (1+t)^{-\frac{\mu}{2}} (1+s) \|u_1\|_{L^1\cap L^2}
%\]
%since $1<\mu<2$.
Suppose that $\max \{  2-\frac{2n}{\sigma} \left ( 1-\frac{1}{q}
\right ) ; 0  \} < \mu < 1$ or $ \mu =2-\frac{2n}{\sigma} \left (
1-\frac{1}{q} \right ) $. We have in $Z_2$
\[\|   u(t,\cdot)\|_{L^q}
\lesssim \,(1+t)^{-\frac{\mu}{2}}
(1+s)^{1+\frac{\mu}{2}-\frac{n}{\sigma}\left(1-\frac1{q}\right)}
\|u_1\|_{L^1} \begin{cases} \left (\ln {\left (
\frac{e+t}{e+s}\right )}\right )^{1-\frac{1}{q}}, \quad  \mu
=2-\frac {2n}{\sigma} \left ( 1-\frac{1}{q}
\right )  \\
1, \quad \max \{  2-\frac{2n}{\sigma} \left ( 1-\frac{1}{q} \right )
; 0  \}  < \mu < 1.
\end{cases} \,
\]
If $\mu \geq 2-\frac{2n}{\sigma} \left ( 1-\frac{1}{q} \right ),$
then
\[
(1+t)^{1-\mu-\frac{n}{\sigma}\left(1-\frac1{q}\right)}(1+s)^{\mu}
\leq (1+t)^{-\frac{\mu}{2}}
(1+t)^{-\frac{\mu}{2}+1-\frac{n}{\sigma}\left(1-\frac1{q}\right)}
(1+s)^{\mu} \leq (1+t)^{-\frac{\mu}{2}}
(1+s)^{\frac{\mu}{2}+1-\frac{n}{\sigma}\left(1-\frac1{q}\right)}.
\]
Hence, we have in $Z_3$
\begin{eqnarray*}
\|   u(t,\cdot)\|_{L^q}   & \lesssim &  (1+t)^{1-\mu-\frac{n}{\sigma}\left(1-\frac1{q}\right)}(1+s)^{\mu} \|u_1\|_{L^1}\\
 & \lesssim & (1+t)^{-\frac{\mu}{2}} (1+s)^{1+\frac{\mu}{2}-\frac{n}{\sigma}\left(1-\frac1{q}\right)}
\|u_1\|_{L^1}.
\end{eqnarray*}
%and in $Z_1$ we have
%\begin{eqnarray*}
%% \nonumber to remove numbering (before each equation)
%  \| u(t,\cdot)\|_{L^q} &  \lesssim &  (1+t)^{-\frac{\mu}2}(1+s)^{1+\frac{\mu}2-\frac{n}{\sigma}\left(\frac12-\frac1{q}\right)}
%          \|u_1\|_{L^2} \\
%   &  \lesssim & (1+t)^{-\frac{\mu}{2}} (1+s)^{
%1+\frac{\mu}{2}-\frac{n}{\sigma}\left(1-\frac1{q}\right)+
%\frac{n}{2\sigma}} \|u_1\|_{L^2}
%\end{eqnarray*}
The proof of (iii): Suppose that $0 < \mu < \min
\{ 2-\frac{2n}{\sigma} \left ( 1-\frac{1}{q} \right
  ) ; 1 \}$. If $ \mu < 2-\frac{2n}{\sigma} \left ( 1-\frac{1}{q} \right
  )$, then
\[
(1+t)^{-\frac{\mu}{2}}(1+s)^{\frac{\mu}{2}}\leq (1+t)^{1-\mu-\frac{n}{\sigma}\left(1-\frac1{q}\right)}(1+s)^{\mu-1+\frac{n}{\sigma}\left(1-\frac1{q}\right)}.
\]
For $n < 2 \sigma$ we get
\[ \|u^{high}(t,s,\cdot)\|_{L^q} \leq C (1+t)^{1-\mu-\frac{n}{\sigma}\left(1-\frac1{q}\right)}(1+s)^{\mu-1+\frac{n}{\sigma}\left(1-\frac1{q}\right)}
\|u_1\|_{L^2}, \quad q \geq 2.\] We obtain the following estimates:
in $ Z_2 \cup Z_3$
\[\|   u(t,\cdot)\|_{L^q}
\lesssim\,(1+t)^{1-\mu-\frac{n}{\sigma}\left(1-\frac1{q}\right)}
(1+s)^{\mu} \|u_1\|_{L^1}
\]
and in $Z_1$
\begin{eqnarray*}
 \| u(t,\cdot)\|_{L^q} & \lesssim &  (1+t)^{-\frac{\mu}{2}} (1+s)^{1+\frac{\mu}2-\frac{n}{\sigma}\left(\frac12-\frac1{q}\right)}
          \|u_1\|_{L^2}\\
   & \lesssim & (1+t)^{1-\mu-\frac{n}{\sigma}\left(1-\frac1{q}\right)}(1+s)^{\mu+\frac{n}{2\sigma}} \|u_1\|_{L^2}
\end{eqnarray*}
or
\begin{eqnarray*}
\|   \mathfrak{F}^{-1}(\chi(|\xi|)\chi_1(s,\xi) \psi(t,s,\xi))\ast
u_1\|_{L^q}
            & \lesssim&\, (1+t)^{1-\mu-\frac{n}{\sigma}\left(1-\frac1{q}\right)}
 (1+s)^{\max\left\{\mu ,
\mu-1+\frac{n}{\sigma}\left(1-\frac1{q}\right)\right\}}\|u_1\|_{L^1}.
             \end{eqnarray*}

Hence we obtain
\[\| u(t,\cdot)\|_{L^q}
\lesssim \, (1+t)^{1-\mu-\frac{n}{\sigma}\left(1-\frac1{q}\right)}
\left( (1+s)^{\mu} \|u_1\|_{L^1}
 +(1+s)^{\mu+\frac{n}{2\sigma}}\|u_1\|_{L^2}
\right)
\]
 or
\[\| u(t,\cdot)\|_{L^q}
\lesssim \, (1+t)^{1-\mu-\frac{n}{\sigma}\left(1-\frac1{q}\right)}
\left( (1+s)^{\max\left\{\mu ,
\mu-1+\frac{n}{\sigma}\left(1-\frac1{q}\right)
\right\}}\|u_1\|_{L^1} +
(1+s)^{\mu-1+\frac{n}{\sigma}\left(1-\frac1{q}\right)}\|u_1\|_{L^2}
\right).
\]
The proof of (iv): Suppose that $\mu=1$. We have in $Z_3$
\[\|   u(t,\cdot)\|_{L^q}
\lesssim\,(1+t)^{-\frac{n}{\sigma}\left(1-\frac1{q}\right)} (1+s)
\ln{\left ( \frac{e+t}{e+s}\right )} \|u_1\|_{L^1} ,
\]
in $Z_2$
\[\|   u(t,\cdot)\|_{L^q}
\lesssim\, \|u_1\|_{L^1}
   \begin{cases}
    (1+t)^{-\frac12}(1+s)^{-\frac{n}{\sigma}\left(1-\frac1{q}\right)+\frac32}\ln{\left ( \frac{e+t}{e+s}\right )}, \quad q > \frac{2n}{[2n-\sigma]_+} \\
    (1+t)^{-\frac12}(1+s)\ln^{2-\frac{1}{q}} {\left ( \frac{e+t}{e+s}\right )}, \quad q = \frac{2n}{[2n-\sigma]_+}\\
    (1+t)^{-\frac{n}{\sigma}\left(1-\frac1{q}\right)}(1+s)\ln{\left ( \frac{e+t}{e+s}\right )}, \quad
    q < \frac{2n}{[2n-\sigma]_+}
    \end{cases}\,,
\]
and in $Z_1$
\[\| u(t,\cdot)\|_{L^q}
          \lesssim\, (1+t)^{-\frac{1}2}(1+s)^{\frac{3}2-\frac{n}{\sigma}\left(\frac12-\frac1{q}\right)}
          \|u_1\|_{L^2}.
          \]
\end{proof}
\begin{corollary}\label{corolinear}
Let $\sigma > 1$, $1 \leq n < 2 \sigma$,  and $q\geq 2$.
\begin{description}
  \item[(i)] If $\mu> \max \{ \frac{2n}{\sigma} \left ( 1-\frac{1}{q} \right
  ); 1 \}$, then the solution to \eqref{mainlineq} with $s=0$ satisfies
\[\|   u(t,\cdot)\|_{L^q}
\lesssim\,  (1+t)^{ -\frac{n}{\sigma}\left(1-\frac1{q}\right)  }  \|u_1\|_{L^1 \cap L^2};
\]
\item[(ii)]
If $\max \{  2-\frac{2n}{\sigma} \left ( 1-\frac{1}{q} \right
  ) ; 0  \} < \mu < 1$ or $1 < \mu < \frac{2n}{\sigma} \left ( 1-\frac{1}{q} \right ) $, then
the solution to \eqref{mainlineq} with $s=0$ satisfies
\[\| u(t,\cdot)\|_{L^q}
\lesssim \, (1+t)^{-\frac{\mu}{2}} \|u_1\|_{L^1 \cap L^2};
\]
Moreover, if $\mu=\frac{2n}{\sigma} \left (1-\frac{1}{q} \right )
>1$, then
 \[\| u(t,\cdot)\|_{L^q}
\lesssim \,
(1+t)^{-\frac{n}{\sigma}\left(1-\frac1{q}\right)}\left(
\left(\ln(e+t)\right)^{1-\frac{1}{q}}
 \|u_1\|_{L^1}
+ \|u_1\|_{L^2} \right)
\]
whereas if $ \mu = 2-\frac{2n}{\sigma} \left ( 1-\frac{1}{q} \right
  )<1, $ then
 \[\| u(t,\cdot)\|_{L^q}
\lesssim \,
(1+t)^{\frac{n}{\sigma}\left(1-\frac1{q}\right)-1}\left(
\left(\ln(e+t)\right)^{1-\frac{1}{q}}
 \|u_1\|_{L^1}
+\|u_1\|_{L^2} \right);
\]

\item[(iii)]
 If $ 0 < \mu < \min \{ 2-\frac{2n}{\sigma} \left
( 1-\frac{1}{q} \right
  ) ; 1 \}, $
then the solution to \eqref{mainlineq} with $s=0$ satisfies
\[\| u(t,\cdot)\|_{L^q}
\lesssim \,
(1+t)^{1-\mu-\frac{n}{\sigma}\left(1-\frac1{q}\right)}
 \|u_1\|_{L^1 \cap L^2};
\]

\item[(iv)]  If $\mu=1$,  then the solution to
\eqref{mainlineq} with $s=0$ satisfies
\begin{eqnarray*}
 \| u(t, \cdot)\|_{L^q}  & \lesssim & (1+t)^{-\frac{1}{2}} \|u_1\|_{L^2} \\
   &+ & \|u_1\|_{L^1} (1+t)^{-\min \{ \frac{n}{\sigma}\left(
1-\frac1q\right); \frac{1}{2} \} }
\begin{cases}
\ln(e+t) , \quad  q \neq \frac{2n}{[2n-\sigma]_+}
\\
(\ln(e+t))^{2-\frac{1}{q}}, \quad  q =  \frac{2n}{[2n-\sigma]_+}.
\end{cases}
\end{eqnarray*}
\end{description}
\end{corollary}
%%%%%%%%%%%%%%%%%%%%%%%%%%%%%%%%%%%

%%%%%%%%%%%%%%%%%%%%%%%%%%%%%%%%%%%%
%\begin{theorem}\label{energy} Let $j,k\in \N$ such that $ j+ k=1$ we have
%\[\| \partial_t^k(- \Delta)^{\frac{j\sigma}{2}}  u(t,\cdot)\|_{L^2}
%\lesssim\, \|u_1\|_{ L^2}\begin{cases} (1+s)^{ \frac{\mu}{2}}(1+t)^{ -\frac{\mu}{2}}, \qquad \mu \in ]0, 2], \\
%(1+s)(1+t)^{-1}, \qquad \mu\geq 2.
%\end{cases}
%\]
%\end{theorem}
%\begin{proof}  The   decay in zones $Z_{high}$ and $Z_1$  is given by $(1+s)^{ \frac{\mu}{2}}(1+t)^{ -\frac{\mu}{2}}$. The same is true in $Z_2$ for $\mu \in (0,2)$,  whereas  it is $(1+t)^{-1}(1+s)$ for $\mu\geq 2$ in  $Z_2$  and in $Z_3$ for  $\mu>1$.
%Now,  $\mu\geq 2$,   using that
%\[(1+t)^{-\frac{\mu}{2}}(1+s)^{\frac{\mu}{2}}\leq (1+t)^{-1}(1+s)\]
%we conclude the desired estimate for $\mu>1$. For $\mu\leq 1$ the decay in $Z_3$ is given by $(1+s)^{ \mu}(1+t)^{ -\mu}$ and the result is conclude.\end{proof}

%%%%%%%%%%%%%%%%%%%%%%%%%%%%%%%%%%%%
%%%%%%%%%%%%%%%%%%%%%%%%%%%%%%%%%%%%%
%%%%%%%%%%%%%%%%%%%%%%%%%%%%%%%%%%%%%%
We complete this section with some energy estimates:
\begin{theorem}\label{fractionalderivatives} Let $\sigma > 1$, $1 \leq n < 2 \sigma$ and $\gamma\in [0, \sigma]$.
    \begin{description}
        \item[(i)] If $\mu> \max \{ \frac{n+2\gamma}{\sigma}; 1 \}$, then the solution to \eqref{mainlineq}  satisfies
        \[\|   u(t,\cdot)\|_{\dot{H}^\gamma}
        \lesssim\,
        (1+t)^{-\frac{n}{2\sigma}-\frac{\gamma}{\sigma}}(1+s)\Big(\|u_1\|_{ L^1}+ (1+s)^{\frac{n}{2\sigma} }\|u_1\|_{ L^2}\Big);
        \]
        \item[(ii)] If $\max \{ 2-\frac{n}{\sigma} -\frac{2\gamma}{\sigma} ; 0 \} < \mu < 1$ or $1< \mu < \frac{n+2\gamma}{\sigma},$ then the solution to \eqref{mainlineq}  satisfies
        \[\|   u(t,\cdot)\|_{\dot{H}^\gamma}
        \lesssim\,
        (1+t)^{-\frac{\mu}{2}}(1+s)^{1+\frac{\mu}{2} -\frac{n}{2\sigma}-\frac{\gamma}{\sigma}}
        \Big(\|u_1\|_{ L^1}+ (1+s)^{\frac{n}{2\sigma} }\|u_1\|_{ L^2}\Big),
        \]
        Moreover, if $\mu=\frac{n+2\gamma}{\sigma}
        >1$, then
        \[\| u(t,\cdot)\|_{\dot{H}^\gamma}
        \lesssim \,
        (1+t)^{-\frac{\mu}{2}}(1+s) \Big( \Big(\ln\Big(\frac{e+t}{e+s}\Big)\Big)^{\frac{1}{2}} \|u_1\|_{ L^1}+ (1+s)^{\frac{n}{2\sigma} }\|u_1\|_{ L^2}\Big),
        \]
        whereas if $ \mu = 2-\frac{n}{\sigma} -\frac{2\gamma}{\sigma}<1$, then
        \[\| u(t,\cdot)\|_{\dot{H}^\gamma}
        \lesssim \,
        (1+t)^{ \frac{n}{2\sigma}-1-\frac{\gamma}{\sigma}}(1+s)^{\mu}
        \Big(\ln {\Big ( \frac{e+t}{e+s}\Big) }\Big)^{\frac{1}{2}}
        \|u_1\|_{L^1}
        +(1+s)^{\frac{n}{2\sigma} } \|u_1\|_{L^2}\Big);
        \]
        \item[(iii)] If  $ 0 < \mu < \min \{ 2-\frac{n}{\sigma} -\frac{2\gamma}{\sigma} ; 1 \} $, then the solution to \eqref{mainlineq}  satisfies
        \[\|   u(t,\cdot)\|_{\dot{H}^\gamma}
        \lesssim\,
        (1+t)^{1-\mu-\frac{n}{2\sigma}-\frac{\gamma}{\sigma}}(1+s)^\mu \Big(\|u_1\|_{ L^1}+ (1+s)^{\frac{n}{2\sigma} }\|u_1\|_{ L^2}\Big).
        \]
    %   \item[(iv)] If $\mu=1$, $\cdots$
    \end{description}   Moreover, the $ \| \partial_t u(t,\cdot)\|_{L^2}$ satisfies the same decay estimates  of $\| (- \Delta)^{\frac{\sigma}{2}}  u(t,\cdot)\|_{L^2}$.
    \end{theorem}
\begin{proof}
    For $n < 2 \sigma$, from Corollary \ref{corHigh} we get
        \[ \||\xi|^{\gamma} u^{high}(t,s,\cdot)\|_{L^2} \leq C
    (1+t)^{-\frac{\mu}{2}}(1+s)^{\frac{\mu}{2}}\|u_1\|_{L^2}, \quad \gamma\in[0, \sigma].\]
        Putting $q=2$, $j\sigma = \gamma$, $k=0$ and $\frac1r = \frac1{2}-\frac1{p'}=\frac1p-\frac12$  and following the calculations on Section \ref{lplqestimates} we get:\\
%    In $Z_3$ we have
%   for $\rho \neq 0$ and $1\leq p\leq 2$, the estimate,
%    \begin{align*}
%    \|   \mathfrak{F}^{-1}(\chi_3(t, s,\xi)|\xi|^{\gamma}
%    \psi(t,s,\xi))\ast u_1\|_{L^2}
%    &\lesssim\,(1+t)^{-\frac{n}{\sigma}\left(\frac1p-\frac12\right)-\frac{\gamma}{\sigma}+\rho+|\rho|}(1+s)^{1-\rho-|\rho|}
%    \|u_1\|_{L^p}\,
%    \end{align*}
%%If $\rho \in \mathbb{Z}$ and $\rho \neq 0$ it is possible to derive the same estimatives as the previous case.
%\\
%    In $Z_2$ we have for $\rho \neq 0$ and $1\leq p\leq 2$, the estimate
%%
%\begin{align*}
%\|    \mathfrak{F}^{-1}(\chi_2(t, s,\xi)|\xi|^{\gamma}
%\psi(t,s,\xi))\ast u_1\|_{L^2}
%&\lesssim  (1+s)^{1-\rho-|\rho|}(1+t)^{\rho-\frac12}\|u_1\|_{L^p} \\
%&    \times  \begin{cases}
%(1+s)^{-\frac{n}{r\sigma}+|\rho|+\frac12-\frac{\gamma}{\sigma}}, \quad  n + r\gamma>r\sigma(|\rho|+\frac12) \\
%\ln^{\frac{1}{r}} {\left ( \frac{e+t}{e+s}\right )}, \quad n + r\gamma=r\sigma(|\rho|+\frac12) \\
%(1+t)^{-\frac{n}{r\sigma}+|\rho|+\frac12-\frac{\gamma}{\sigma}}, \quad
%n + r\gamma<r\sigma(|\rho|+\frac12)
%\end{cases}\,.
%\end{align*}
%    \\
%In $Z_1$ we have for $1\leq p\leq 2$, the estimate
%%
%\begin{align*}
%\|
%\mathfrak{F}^{-1}(\chi(|\xi|)\chi_1(s,\xi)|\xi|^{\gamma}
%\psi(t,s,\xi))\ast u_1\|_{L^2}
%& \lesssim\, (1+s)^{\frac{\mu}2}(1+t)^{-\frac{\mu}2}\|u_1\|_{L^p} \\
%&    \times  \begin{cases}
%1, \quad r(-\sigma+\gamma)+n>0 \\
%\ln^{\frac{1}{r}}{(e+s)}, \quad r(-\sigma+\gamma)+n=0\\
%(1+s)^{1-\frac{\gamma}{\sigma}-\frac{n}{r\sigma}}, \quad r(-\sigma+\gamma)+n<0
%\end{cases}\,.
%\end{align*}
     The proof of (i): Suppose $\mu> \max \{ \frac{n+2\gamma}{\sigma}; 1 \}$.
    In $Z_3$ we have
    for $ p=1$, the estimate
    \begin{align*}
    \|   \mathfrak{F}^{-1}(\chi_3(t, s,\xi)|\xi|^{\gamma}
    \psi(t,s,\xi))\ast u_1\|_{L^2}
    &\lesssim\,(1+t)^{-\frac{n}{2\sigma}-\frac{\gamma}{\sigma}}(1+s)
    \|u_1\|_{L^1}.
    \end{align*}
        In $Z_2$ we have
        for $ p=1$, the estimate
        \begin{align*}
        \|   \mathfrak{F}^{-1}(\chi_2(t, s,\xi)|\xi|^{\gamma}
        \psi(t,s,\xi))\ast u_1\|_{L^2}
        &\lesssim\,(1+t)^{-\frac{n}{2\sigma}-\frac{\gamma}{\sigma}}(1+s)
        \|u_1\|_{L^1}.
        \end{align*}
        In $Z_1$ we have
        for $ p=2$, the estimate
        \begin{align*}
        \|   \mathfrak{F}^{-1}(\chi(|\xi|) \chi_1(t, s,\xi)|\xi|^{\gamma}
        \psi(t,s,\xi))\ast u_1\|_{L^2}
        &\lesssim\,(1+t)^{-\frac{\mu}{2}}(1+s)^{\max\{ \frac{\mu}{2}, \frac{\mu}{2}+1-\frac{\gamma}{\sigma}\} }
        \|u_1\|_{L^2}\\
        &\lesssim\,(1+t)^{-\frac{n}{2\sigma}-\frac{\gamma}{\sigma}}(1+s)^{\max\{ \frac{n}{2\sigma}+\frac{\gamma}{\sigma}, 1+\frac{n}{2\sigma}\} }
        \|u_1\|_{L^2}.
        \end{align*}
         The proof of (ii):
Suppose that $1 < \mu \leq
\frac{n+2\gamma}{\sigma}$.
    In $Z_1$ we have
    for $ p=2$, the estimate
    \begin{align*}
    \|   \mathfrak{F}^{-1}(\chi(|\xi|)\chi_1(t, s,\xi)|\xi|^{\gamma}
    \psi(t,s,\xi))\ast u_1\|_{L^2}
    &\lesssim\,(1+t)^{-\frac{\mu}{2}}(1+s)^{\max\{ \frac{\mu}{2}, \frac{\mu}{2}+1-\frac{\gamma}{\sigma}\} }
    \|u_1\|_{L^2}.
    \end{align*}
We have in $Z_2$
\[\|   \mathfrak{F}^{-1}(\chi_2(t, s,\xi)|\xi|^{\gamma}
\psi(t,s,\xi))\ast u_1\|_{L^2}
\lesssim\,(1+t)^{-\frac{\mu}{2}}
(1+s)^{1+\frac{\mu}{2}-\frac{n}{2\sigma}-\frac{\gamma}{\sigma}} \|u_1\|_{L^1}
\begin{cases} \left(\ln {\left (
    \frac{e+t}{e+s}\right )}\right)^{
    \frac{1}{2}}, \quad  \mu=\frac{n+2\gamma}{\sigma} \\
1, \quad  1<\mu< \frac{n+2\gamma}{\sigma}
\end{cases} \,
\]
If  $\mu \leq \frac{n+2\gamma}{\sigma}$, then we
have in $Z_3$
\begin{eqnarray*}
\|   \mathfrak{F}^{-1}(\chi_3(t, s,\xi)|\xi|^{\gamma}
\psi(t,s,\xi))\ast u_1\|_{L^2}
    & \lesssim & (1+t)^{-\frac{n}{2\sigma}-\frac{\gamma}{\sigma}}(1+s)\|u_1\|_{L^1}\\
        & \lesssim &
    (1+t)^{-\frac{\mu}{2}}(1+s)^{1+\frac{\mu}{2}-\frac{n}{2\sigma}-\frac{\gamma}{\sigma}} \|u_1\|_{L^1}.
\end{eqnarray*}
%%%%%%%%%%
Suppose that $\max \{  2-\frac{n}{\sigma} -\frac{2\gamma}{\sigma} ; 0  \} < \mu < 1$ or $ \mu = 2-\frac{n}{\sigma} -\frac{2\gamma}{\sigma}  $. We have in $Z_2$
\[ \|   \mathfrak{F}^{-1}(\chi_2(t, s,\xi)|\xi|^{\gamma}
\psi(t,s,\xi))\ast u_1\|_{L^2}
\lesssim \,(1+t)^{-\frac{\mu}{2}}
(1+s)^{1+\frac{\mu}{2}-\frac{n}{2\sigma}-\frac{\gamma}{\sigma}}
\|u_1\|_{L^1} \begin{cases} \left (\ln {\left (
    \frac{e+t}{e+s}\right )}\right )^{\frac{1}{2}}, \quad  \mu
=2-\frac{n}{\sigma} -\frac{2\gamma}{\sigma}  \\
1, \quad \max \{ 2-\frac{n}{\sigma} -\frac{2\gamma}{\sigma}
; 0  \}  < \mu < 1
\end{cases} \,.
\]
If $\mu \geq 2-\frac{n}{\sigma}-\frac{2\gamma}{\sigma}$,
then we have in $Z_3$
\begin{eqnarray*}
    \|   \mathfrak{F}^{-1}(\chi_3(t, s,\xi)|\xi|^{\gamma}
\psi(t,s,\xi))\ast u_1\|_{L^2}  & \lesssim &  (1+t)^{1-\mu-\frac{n}{2\sigma}-\frac{\gamma}{\sigma}}(1+s)^{\mu} \|u_1\|_{L^1}\\
    & \lesssim & (1+t)^{-\frac{\mu}{2}}
    (1+s)^{\frac{\mu}{2}+1-\frac{n}{2\sigma}-\frac{\gamma}{\sigma}}
    \|u_1\|_{L^1}.
\end{eqnarray*}
The proof of (iii):
Suppose that $0 < \mu < \min
\{  2-\frac{n}{\sigma}-\frac{2\gamma}{\sigma}; 1 \}$. If $ \mu <  2-\frac{n}{\sigma}-\frac{2\gamma}{\sigma},$ then for $n < 2 \sigma$ we get
    \[ \||\xi|^{\gamma} u^{high}(t,s,\cdot)\|_{L^2} \leq C(1+t)^{-\frac{\mu}{2}}(1+s)^{\frac{\mu}{2}}\|u_1\|_{L^2} \leq C (1+t)^{1-\mu-\frac{n}{2\sigma}-\frac{\gamma}{\sigma}}(1+s)^{\mu-1+\frac{n}{2\sigma}+\frac{\gamma}{\sigma}} \|u_1\|_{L^2}, \quad q \geq 2.\] We obtain the following estimates:
in $ Z_2 \cup Z_3$
\[ \|   \mathfrak{F}^{-1}(\chi_2(t, s,\xi)\chi_3(t, s,\xi)|\xi|^{\gamma}
\psi(t,s,\xi))\ast u_1\|_{L^2}
\lesssim\,(1+t)^{1-\mu-\frac{n}{2\sigma}-\frac{\gamma}{\sigma}}
(1+s)^{\mu} \|u_1\|_{L^1}
\]
and in $Z_1$
\begin{eqnarray*}
    \|   \mathfrak{F}^{-1}(\chi(|\xi|)\chi_1(t, s,\xi)\chi_3(t, s,\xi)|\xi|^{\gamma}
    \psi(t,s,\xi))\ast u_1\|_{L^2} & \lesssim &  (1+t)^{-\frac{\mu}{2}} (1+s)^{\max\{ \frac{\mu}{2}, \frac{\mu}{2}+1-\frac{\gamma}{\sigma}\} }
    \|u_1\|_{L^2}\\
    & \lesssim & (1+t)^{1-\mu-\frac{n}{2\sigma}-\frac{\gamma}{\sigma}}(1+s)^{\max\{ \mu-1+\frac{n}{2\sigma}+ \frac{\gamma}{\sigma}, \mu + \frac{n}{2\sigma}\}} \|u_1\|_{L^2}.
\end{eqnarray*}

\end{proof}
\section{Proof of the Global existence results}\label{NL}
By Duhamel's principle, a function $u\in Z$, where~$Z$ is a suitable space, is a solution to~\eqref{eq:DPE} if, and only if, it satisfies the equality
\begin{equation}\label{eq:fixedpoint}
u(t,x) =  u^{\mathrm{lin}}(t,x) + \int_0^t K_1(t, s, x) \ast \,|u(s,x)|^p\, ds\,, \qquad \text{in~$Z$,}
\end{equation}
where $K_1(t, s, x)=\mathscr{F}^{-1}(\psi)(t, s,
x)$ and
\[ u^{\mathrm{lin}} (t,x) \doteq  K_1(t, 0, x) \ast u_1(x) \,, \]
is the solution to the linear Cauchy problem~\eqref{mainlineq} with
$s=0$. % \textcolor{green}{and $K_1(t, 0, x)$ is the fundamental solution of \eqref{mainlineq}, that is, the distributional solution
%with initial data $(u_0,u_1)=(0,\delta_0)$, $\delta_0$ denotes the
%Dirac measure at point $x=0$}.
%
The proof of our global existence results is based on the following scheme. We  define an appropriate data function space
\begin{align}
\label{dom:A}
\mathcal A  \doteq  L^2(\R^n)\cap L^{1}(\R^n),
\end{align}
 and an evolution space for solutions
\begin{align}
\label{eq:Xsp} Z(T) \doteq C([0, T], H^{\sigma}(\R^n)) \cap C^1([0,
T], L^2(\R^n))\cap L^{\infty}([0, \infty)\times\mathbf{R}^n),
\end{align}
equipped with a  norm relate to the estimates  of solutions to the linear problem \eqref{mainlineq} with $s=0$ such that
\begin{equation}\label{eq:ubasic}
\|u^{\mathrm{lin}}(t, \cdot)\|_{Z} \leq C\,\| u_1\|_{\mathcal A}.
\end{equation}

We define the operator~$F$ such that, for any~$u\in Z$,
\begin{equation}\label{eq:G}\nonumber
Fu(t,x) \doteq \int_0^t K_1(t, s,x)\ast |u(s,x)|^p\, ds\,,
\end{equation}
then we prove the estimates
\begin{align}
\label{eq:well}
\|Fu\|_{Z}
    & \leq C\|u\|_{Z}^p\,, \\
\label{eq:contraction}
\|Fu-Fv\|_{Z}
    & \leq C\|u-v\|_{Z} \bigl(\|u\|_{Z}^{p-1}+\|v\|_{Z}^{p-1}\bigr)\,.
\end{align}
By standard arguments, since $u^{\mathrm{lin}}$ satisfies~\eqref{eq:ubasic} and~$p>1$, from~\eqref{eq:well} it follows that~$u^{\mathrm{lin}}+F$ maps balls of~$Z$ into balls of~$Z$, and for small data in $\mathcal{A},$ from \eqref{eq:contraction} $F$ is a contraction. So, the estimates \eqref{eq:well}-\eqref{eq:contraction} lead to the existence of a unique solution to~\eqref{eq:fixedpoint}, that is, $u=u^{\mathrm{lin}}+Fu$, satisfying~\eqref{eq:ubasic}. We simultaneously gain a locally in time  for large data  and
 globally in time for small data existence result \cite{ER}.

\medskip

%The information that~$u\in Z$ plays a fundamental role to estimate~$f(u_t(s,\cdot))$ in suitable norms. We will employ the following well-known result (for instance, see \cite{DAE17NA}).
%%
%\begin{lemma}\label{lem:integral}
%Let~$\kappa\leq 1$. Then it holds
%%
%\[%\begin{equation}\label{eq:integralest}
%\int_0^t (1+t-s)^{-\kappa}\,\,(1+s)^{-\beta}\,ds \lesssim \left\{ \begin{array}{ll} (1+t)^{-\kappa} \hspace{2cm} \text{if } \beta >1\\ (1+t)^{-\kappa} \log(e+t)\hspace{0,5cm} \text{if } \beta =1.\end{array}\right.
%\]%\end{equation}
%%
%\end{lemma}
\begin{proof}(Theorem \ref{theoremN1})
We have to prove \eqref{eq:ubasic}, \eqref{eq:well} and \eqref{eq:contraction}, with $\mathcal{A}$ as in \eqref{dom:A} and $Z(T)$ as in \eqref{eq:Xsp}  equipped with the norm
\begin{align*}
\nonumber
\|u\|_{Z(T)}
   & \doteq \sup_{t\in[0,T]} \Bigl\{
    (1+t)^{\frac{n}{2\sigma}}\|u(t,\cdot)\|_{L^2} + (1+t)^{\frac{n}{ \sigma}\left(1-\frac1{q_0}\right)} \|u(t,\cdot)\|_{L^{q_0}}
    + (1+t)^{\min\left\{\frac{n}{\sigma}, \frac{\mu}{2 }\right\}}\|u(t,\cdot)\|_{L^{\infty}} \\
      &
         + (1+t)^{\min\{\frac{n}{2\sigma}+1, \frac{\mu}{2}\}} \Big(\|u_t(t,\cdot)\|_{L^{2}}
    +\|   u(t,\cdot)\|_{\dot{H}^\sigma}\Big) %+\textcolor{red}{(1+t)^{\frac{\mu}{2}}\|   u(t,\cdot)\|_{\dot{H}^{\frac{\sigma\mu-n}{2}}}}
     \Bigr\},
\end{align*}
where $q_0$ is defined as in \eqref{fujita}.\\
Thanks to Corollary  \ref{corolinear} and Theorem \ref{fractionalderivatives}, $u^{\mathrm{lin}} \in Z(T)$ and it satisfies \eqref{eq:ubasic}.

Let us prove \eqref{eq:well}. We omit the proof of~\eqref{eq:contraction}, since it is analogous to the proof of~\eqref{eq:well}.

Let $u\in Z(T)$. If $\mu > \max \{ \frac{2n}{\sigma} ; 1
\}$, by Theorem \ref{theoremlinear}, for $q\geq
2$ we have
 \begin{eqnarray*}
  \| Fu(t, \cdot)\|_{L^q}
 & \lesssim &\int_0^t (1+t)^{-\frac{n}{\sigma}\left(1-\frac1{q}\right) }(1+s)\left(\| |u(s, \cdot)|^p\|_{L^1} +(1+s)^{ \frac{n}{2\sigma}} \||u(s, \cdot)|^p\|_{L^2}\right) ds\\
 &\lesssim & (1+t)^{-\frac{n}{\sigma}\left(1-\frac1{q}\right) }\int_0^t \left( (1+s)^{1 - \frac{n}{\sigma}\left(1-\frac1{p}\right)p}+
 (1+s)^{1+\frac{n}{2\sigma} -\frac{n}{\sigma}\left(1-\frac1{2p}\right)p}  \right)ds\| u\|_{Z(T)}^p\\
 &\lesssim &(1+t)^{-\frac{n}{\sigma}\left(1-\frac1{q}\right) }\| u\|_{Z(T)}^p,\end{eqnarray*}
 for all $p>  1+
\frac{2\sigma}{n}$, that is, $ \frac{n}{\sigma}(p-1)-1>1$ and
\[ \frac{n}{\sigma}\left(p-\frac12\right)-1-\frac{n}{2\sigma}>1. \]

If $ \max \left\{ \frac{n}{\sigma}+ \frac{2n}{n+2\sigma} ; 1 \right\}<\mu <\frac{2n}{\sigma}$ and $p\leq  \frac{n}{2n-\sigma\mu}$, then
 %by Remark \ref{hypothesis} we have that
 $\mu\geq \frac{2n}{
\sigma}\left(1-\frac1{2p}\right) $, hence $L^q$ norm of $u$, with $2\leq q\leq q_0$ may be estimate as in the previous case, whereas
\begin{eqnarray*}
  \| Fu(t, \cdot)\|_{L^\infty}
 & \lesssim &\int_0^t (1+t)^{-\frac{\mu}{2} }(1+s)^{
1+\frac{\mu}{2}-\frac{n}{\sigma} }\left(\| |u(s, \cdot)|^p\|_{L^1} +(1+s)^{ \frac{n}{2\sigma}} \||u(s, \cdot)|^p\|_{L^2}\right) ds\\
  & \lesssim & (1+t)^{-\frac{\mu}{2} }\int_0^t(1+s)^{
1+\frac{\mu}{2}-\frac{n}{\sigma} } \left((1+s)^{ - \frac{n}{\sigma}\left(1-\frac1{p}\right)p}+ (1+s)^{\frac{n}{2\sigma} -\frac{n}{\sigma}\left(1-\frac1{2p}\right)p}  \right) ds\| u\|_{Z(T)}^p
 \\
 &\lesssim & (1+t)^{-\frac{\mu}{2}}\int_0^t (1+s)^{
    1+\frac{\mu}{2}-\frac{np}{\sigma}} ds\| u\|_{Z(T)}^p
 \lesssim (1+t)^{-\frac{\mu}{2}}\| u\|_{Z(T)}^p,
 \end{eqnarray*}
for all $p>  1+ \frac{2\sigma}{n}>\frac{\mu\sigma}{2n}+
\frac{2\sigma}{n}$.

 Finally, if $\mu>  \max \left\{ \frac{n}{\sigma}+ \frac{2n}{n+2\sigma} ; 1 \right\}$, by Theorem
\ref{fractionalderivatives}   we have
%for $\gamma\in (0, \sigma]$
\begin{eqnarray*}
  \|  Fu(t, \cdot)\|_{\dot{H}^\sigma}
 & \lesssim & (1+t)^{-\min\left\{\frac{n}{2\sigma}+1, \frac{\mu}{2 }\right\}}\int_0^t (1+s)\left(\| |u(s, \cdot)|^p\|_{L^1} +(1+s)^{ \frac{n}{2\sigma}} \||u(s, \cdot)|^p\|_{L^2}\right) ds\\
 &\lesssim & (1+t)^{-\min\left\{\frac{n}{2\sigma}+1, \frac{\mu}{2 }\right\}}\int_0^t \left((1+s)^{1 - \frac{n}{\sigma}\left(1-\frac1{p}\right)p}+(1+s)^{1+\frac{n}{2\sigma}-\frac{n(2p-1)}{2\sigma}}  \right) ds\| u\|_{Z(T)}^p\\
 &\lesssim &(1+t)^{-\min\left\{\frac{n}{2\sigma}+1, \frac{\mu}{2 }\right\}}\| u\|_{Z(T)}^p,\end{eqnarray*}
 and
 \[ \| \partial_t Fu(t, \cdot)\|_{L^2}\lesssim (1+t)^{-\min\left\{\frac{n}{2\sigma}+1, \frac{\mu}{2 }\right\}}\| u\|_{Z(T)}^p,\]
 for all $p>1+ \frac{2\sigma}{n}$.
 \end{proof}
\begin{proof}(Theorem \ref{theoremNL2})
We have to prove \eqref{eq:ubasic}, \eqref{eq:well} and \eqref{eq:contraction}, with $\mathcal{A}$ as in \eqref{dom:A} and $Z(T)$ as in \eqref{eq:Xsp}  equipped with the norm
\begin{align*}
\nonumber
\|u\|_{Z(T)}
   & \doteq \sup_{t\in[0,T]} \Bigl\{
    (1+t)^{ \frac{n}{2\sigma}+\mu-1} \|u(t,\cdot)\|_{L^2} + (1+t)^{\frac{n}{\sigma}\left(1-\frac1{q_1}\right)-1+\mu}
    \|u(t,\cdot)\|_{L^{q_1}} +(1+t)^{\min\left\{\frac{n}{\sigma}+\mu-1, \frac{\mu}{2}\right\}}||  u (t, \cdot) ||_{L^\infty}
     \\
      &
         + (1+t)^{ \frac{\mu}{2 }} \Big(\|u_t(t,\cdot)\|_{L^{2}}
    +\|(- \Delta)^{\frac{\sigma}{2}}  u(t,\cdot)\|_{L^2} %\textcolor{red}{(1+t)^{\frac{\mu}{2}}\|   u(t,\cdot)\|_{\dot{H}^{\sigma-\frac{n}{2}-\frac{\sigma\mu}{2}}}}
    \Big)\Bigr\}.
\end{align*}
Thanks to Corollary \ref{corolinear} and Theorem
\ref{fractionalderivatives}, $u^{\mathrm{lin}} \in Z(T)$ and it
satisfies \eqref{eq:ubasic}.

Let us prove \eqref{eq:well}. We omit the proof of~\eqref{eq:contraction}, since it is analogous to the proof of~\eqref{eq:well}.

Let $u\in Z(T)$. If $ 1-\frac{n}{\sigma} < \mu <
\min \{ 2-\frac{2n}{\sigma} ; 1 \} $ by Theorem \ref{theoremlinear}
for $q \geq 2$  we have
\begin{eqnarray*}
  \| Fu(t, \cdot)\|_{L^q}
 & \lesssim &\int_0^t (1+t)^{-\frac{n}{\sigma}\left(1-\frac1{q}\right)+1-\mu}(1+s)^{\mu}\left(\| |u(s, \cdot)|^p\|_{L^1}+
(1+s)^{\frac{n}{2\sigma}} \| |u(s, \cdot)|^p\|_{L^2} \right) ds\\
 &\lesssim & (1+t)^{-\frac{n}{\sigma}\left(1-\frac1{q}\right)+1-\mu}
 \int_0^t\left( (1+s)^{\mu - \frac{n}{\sigma}\left(p-1\right)+(1-\mu)p}+(1+s)^{\mu+\frac{n}{2\sigma}- \frac{n}{\sigma}\left(p-\frac12\right)+(1-\mu)p}  \right)ds
 \| u\|_{Z(T)}^p\\
 &\lesssim &(1+t)^{-\frac{n}{\sigma}\left(1-\frac1{q}\right)+1-\mu}\| u\|_{Z(T)}^p,\end{eqnarray*}
 for all $ p> \frac{n+\sigma +\sigma\mu}{n-\sigma+\sigma\mu}$.\\

If $2-\frac{2n}{\sigma} <\mu< \min \left\{ \mu_{\sharp} ; 1 \right\}$ and $2p\leq q_1$,
then $L^q$ norm of $u$, with $2\leq q\leq q_1$ may be estimate as in
the previous case, whereas
\begin{eqnarray*}
  \| Fu(t, \cdot)\|_{L^\infty}
 & \lesssim &\int_0^t (1+t)^{-\frac{\mu}{2} }(1+s)^{
1+\frac{\mu}{2}-\frac{n}{\sigma} }\left(\| |u(s, \cdot)|^p\|_{L^1} +(1+s)^{ \frac{n}{2\sigma}} \||u(s, \cdot)|^p\|_{L^2}\right) ds\\
  & \lesssim & (1+t)^{-\frac{\mu}{2} }\int_0^t(1+s)^{
1+\frac{\mu}{2}-\frac{n}{\sigma} } \left( (1+s)^{ -
\frac{n}{\sigma}\left(1-\frac1{p}\right)p + (1-\mu)p}+
(1+s)^{\frac{n}{2\sigma}
-\frac{n}{\sigma}\left(1-\frac1{2p}\right)p+ (1-\mu)p}  \right) ds\|
u\|_{Z(T)}^p
 \\
 &\lesssim & (1+t)^{-\frac{\mu}{2}}\int_0^t (1+s)^{
    1+\frac{\mu}{2}-\frac{np}{\sigma}+(1-\mu)p} ds\| u\|_{Z(T)}^p
 \lesssim (1+t)^{-\frac{\mu}{2}}\| u\|_{Z(T)}^p,
 \end{eqnarray*}
for all $ p> \frac{n+\sigma +\sigma\mu}{n-\sigma+\sigma\mu} > \frac{2\sigma+\sigma\mu/2}{n-\sigma+\sigma\mu} $.\\
Now, if $p\leq q_1< 2p$ we use the interpolation
\begin{equation}\label{int}
\| u\|_{L^{2p}}\leq \| u\|_{L^{q_1}}^{\theta}\|
u\|_{L^\infty}^{1-\theta}, \quad  \theta=q_1/2p.
\end{equation}
By Theorem \ref{theoremlinear}(iii) for $2\leq q\leq q_1$ we have
\begin{eqnarray*}
  & \| Fu(t, \cdot)\|_{L^q}
  \lesssim \displaystyle\int_0^t (1+t)^{1-\mu-\frac{n}{\sigma}\left(1-\frac1{q}\right) }
 (1+s)^{\mu}\left(\| |u(s, \cdot)|^p\|_{L^1} +(1+s)^{ \frac{n}{\sigma}\left(1-\frac1{q}\right)-1} \||u(s, \cdot)|^p\|_{L^2}\right) ds\\
 &\lesssim (1+t)^{-1+\mu+\frac{n}{\sigma}\left(1-\frac1{q}\right)
} \times \\
 &\displaystyle\int_0^t \left( (1+s)^{\mu - \frac{n}{\sigma}\left(1-\frac1{p}\right)p+(1-\mu)p}+
 (1+s)^{\mu-1+\frac{n}{\sigma}\left(1-\frac1{q}\right)-\frac{n(q_1-1)}{2\sigma}+\frac{(1-\mu)q_1}{2}
 -\frac{\mu}{2}\left(p-\frac{q_1}{2}\right)}  \right)ds\| u\|_{Z(T)}^p\\
&\lesssim
(1+t)^{-1+\mu+\frac{n}{\sigma}\left(1-\frac1{q}\right)
} \| u\|_{Z(T)}^p,
\end{eqnarray*}
 for all $p> p_K(n+\sigma\mu) > 1+ \frac{2}{\mu} $, thanks to
 \[-\frac{n(q_1-1)}{2\sigma}+\frac{(1-\mu)q_1}{2}
 +\frac{q_1\mu}{4}=0\]
 and
 \[\mu-1+\frac{n}{\sigma}\left(1-\frac1{q}\right)-\frac{p\mu}{2}\leq \mu-1+\frac{2-\mu}{2}-\frac{p\mu}{2}=\frac{\mu(1-p)}{2}<-1.\]
To estimate the $\| Fu(t, \cdot)\|_{L^\infty}$ for  $\mu>2-\frac{2n}{\sigma}$ and $ q_1< 2p$, one may use Theorem \ref{theoremlinear}(ii) and apply again \eqref{int}, namely
\begin{eqnarray*}
  \| Fu(t, \cdot)\|_{L^\infty}
 & \lesssim &\int_0^t (1+t)^{-\frac{\mu}{2} }(1+s)^{
\frac{\mu}{2} }\left((1+s)^{
1-\frac{n}{\sigma} }\| |u(s, \cdot)|^p\|_{L^1} + \||u(s, \cdot)|^p\|_{L^2}\right) ds\\
  & \lesssim & (1+t)^{-\frac{\mu}{2}}\| u\|_{Z(T)}^p,
 \end{eqnarray*}
 for all $p> p_K(n+\sigma\mu) > 1+ \frac{2}{\mu} $, thanks to
 \[\frac{\mu}{2}+1-\frac{n}{\sigma} <\mu.\]

Finally, if $1-\frac{n}{\sigma}<\mu < \min \left\{ \mu_{\sharp} ; 1 \right\}$, by Theorem
\ref{fractionalderivatives}   we have
\begin{eqnarray*}
  \|  Fu(t, \cdot)\|_{\dot{H}^\sigma}
 & \lesssim & (1+t)^{- \frac{\mu}{2 }}\int_0^t (1+s)^{\frac{\mu}{2 }- \frac{n}{2\sigma}}\left(\| |u(s, \cdot)|^p\|_{L^1} +(1+s)^{ \frac{n}{2\sigma}} \||u(s, \cdot)|^p\|_{L^2}\right) ds\\
 &\lesssim &(1+t)^{- \frac{\mu}{2 }}\| u\|_{Z(T)}^p,\end{eqnarray*}
  and
 \[ \| \partial_t Fu(t, \cdot)\|_{L^2}\lesssim (1+t)^{- \frac{\mu}{2 }}\| u\|_{Z(T)}^p,\]
 for all $p>p_K(n+\sigma\mu)$.
% \textcolor{red}{ Similarly,  if $2p\leq q_1$ we have
%\begin{eqnarray*}
%\|   u(t,\cdot)\|_{\dot{H}^{\sigma-\frac{n}{2}-\frac{\sigma\mu}{2}}}& \lesssim &
%(1+t)^{-\frac{\mu}{2}}\int_0^t (1+s)^{\mu}\left(\| |u(s, \cdot)|^p\|_{L^1} +(1+s)^{ \frac{n}{2\sigma}} \||u(s, \cdot)|^p\|_{L^2}\right) ds\\
% &\lesssim &(1+t)^{- \frac{\mu}{2 }}\| u\|_{Z(T)}^p,
%  \end{eqnarray*}
%  for all $p>p_K(n+\sigma\mu)$, but if $q_1<2p$
%  \begin{eqnarray*}
%\|   u(t,\cdot)\|_{\dot{H}^{\sigma-\frac{n}{2}-\frac{\sigma\mu}{2}}}& \lesssim &
%(1+t)^{-\frac{\mu}{2}}\int_0^t (1+s)^{\mu}\left(\| |u(s, \cdot)|^p\|_{L^1} +(1+s)^{ \frac{n}{2\sigma}} \||u(s, \cdot)|^p\|_{L^2}\right) ds\\
% &\lesssim &(1+t)^{- \frac{\mu}{2 }}\| u\|_{Z(T)}^p + \cdots,
%  \end{eqnarray*}
% }

\end{proof}

\section{Proof of the Non existence result via test function method}\label{test}

\begin{proof}(Proposition \ref{testfuntion})
Let us multiply \eqref{eq:DPE} by  the function $g(t)=g(0)(1+t)^{\mu}$, with $g(0)>0$, so that
\begin{equation}\label{modified}
(gu)_{tt}+ (- \Delta)^\sigma(gu) -(g'u)_t = g(t)|u|^p.
\end{equation}

We fix a nonnegative, non-increasing, test function~$\varphi\in\mathcal{C}_c^\infty([0,\infty))$ with~$\varphi=1$ in~$[0,1/2]$ and~$\supp\varphi\subset[0,1]$, and a  nonnegative, radial, test function~$\psi\in\mathcal{C}_c^\infty(\R^n)$, such that~$\psi=1$ in the ball~$B_{1/2}$, and~$\supp\psi\subset B_1$. We also assume~$\psi(x)\leq\psi(y)$ when~$|x|\geq|y|$. Here~$B_r$ denotes the ball of radius~$r$, centered at the origin. We may assume  that
\begin{equation}\label{eq:testbnd}
\varphi^{-\frac{p'}{p}}\,\bigl(|\varphi'|^{p'}+|\varphi''|^{p'}\bigr),
\qquad\psi^{-\frac{p'}{p}} \bigl( |\psi|^{p'}+|\Delta^\sigma
\psi|^{p'}\bigr), \qquad \text{are bounded,}
\end{equation}
where~$p'=p/(p-1)$. We remark that the assumption that~$\sigma$
is integer plays a fundamental role
here. Then, for~$R\geq1$, we define:
\begin{equation}\label{eq:testR}
\varphi_R(t) = \varphi(R^{-\sigma}t), \quad \psi_R(x)=\psi(R^{-1}x).
\end{equation}
Let us assume that~$u$ is a (global or local) weak solution
to~\eqref{modified}. Let~$R>0$, and also assume that~$R\leq
T^\sigma$, if~$u$ is a local solution in~$[0,T]\times\R^n$.
Integrating by parts, and recalling that~$u(0,x)\equiv0$
and~$\varphi_R(0)=1$, we obtain
\begin{equation}\label{eq:test0}
\int_0^\infty \int_{\R^n} u
\bigl(g\varphi_R''\psi_R+g'\varphi_R'\psi_R+g\varphi_R(-\Delta)^\sigma\psi_R\bigr)\,dxdt
-g(0)\int_{\R^n} u_1(x)\,\psi_R(x)\,dx = I_R,
\end{equation}
where:
\[ I_R=\int_0^\infty \int_{\R^n} g(t)|u|^p\varphi_R\psi_R\,dxdt. \]
We may now apply Young inequality to estimate:
\begin{multline*}
\int_0^\infty \int_{\R^n} |u| \bigl(g|\varphi_R''|\,\psi_R+ g'|\varphi_R'|\,| \psi_R|+g\varphi_R\,|(-\Delta)^\sigma\psi_R|\bigr)\,dxdt
    \leq \frac1p\,I_R\\
    +\frac1{p'}\,\int_0^\infty \int_{\R^n} (\varphi_R\psi_R)^{-\frac{p'}{p}}g(t)\bigl(|\varphi_R''\psi_R|^{p'}+|\varphi_R(-\Delta)^\sigma \psi_R|^{p'} \bigr)
    + (g\varphi_R\psi_R)^{-\frac{p'}{p}}|g'\varphi_R'
    \psi_R|^{p'}\,dxdt.
\end{multline*}
Due to
\begin{gather*}
\varphi_R'(t) = R^{-\sigma}(\varphi')(R^{-\sigma}t), \qquad \varphi_R''(t) = R^{-2\sigma}(\varphi'')(R^{-\sigma}t),\\
(-\Delta)^\sigma\psi_R(x)=R^{-2\sigma}\bigl((-\Delta)^\sigma\psi\bigr)(R^{-1}x),
\end{gather*}
recalling~\eqref{eq:testbnd}, we may estimate
\begin{align*}
\int_0^\infty \int_{\R^n} g(t)(\varphi_R\psi_R)^{-\frac{p'}{p}} |\varphi_R''\psi_R|^{p'}\,dxdt
    & \leq C\,R^{-2\sigma p'+n+(1+\mu)\sigma},\\
\int_0^\infty \int_{\R^n} (g\varphi_R\psi_R)^{-\frac{p'}{p}} |g'\varphi_R' \psi_R|^{p'}\,dxdt
    & \leq C\,R^{-2\sigma p'+n+(1+\mu)\sigma},\\
\int_0^\infty \int_{\R^n} g(t)(\varphi_R\psi_R)^{-\frac{p'}{p}} |\varphi_R(-\Delta)^\sigma \psi_R|^{p'}\,dxdt
    & \leq C\,R^{-2\sigma p'+n +(1+\mu)\sigma}.
\end{align*}
Summarizing, we proved that
\[ \frac1{p'}\,I_R\leq C\,R^{-2\sigma p'+n +(1+\mu)\sigma} - g(0) \int_{\R^n} u_1(x)\,\psi_R(x)\,dx .\]
Assume, by contradiction, that the solution~$u$ is global
(in time). Recalling assumption~\eqref{blowup},
in the subcritical case~$p<\frac{n+\sigma
+\sigma\mu}{[n-\sigma+\sigma\mu]_+}$, it follows that~$I_R<0$, for
any sufficiently large~$R$, and this contradicts the fact
that~$I_R\geq0$. The critical
case~$p=\frac{n+\sigma
+\sigma\mu}{[n-\sigma+\sigma\mu]_+}$ is treated in standard way,
but we omit the details for the sake of brevity. Therefore, $u$
cannot be a global solution (in time).
\end{proof}

\section*{Appendix}\label{sec:Appendix}

In this section we  include notations,
 well known results of Harmonic Analysis
and properties of special functions used throughout the paper.

\begin{notation}
By~$[x]_+$ we denote the non-negative
part of~$x\in \mathbf{R}$, i.e. $[x]_+=\max\{x,0\}$.
\end{notation}

\begin{notation}
We write $ f\lesssim g$ if there exists a
constant $C>0$ such that $f\leq C g$, and $f\approx g$ if $g\lesssim f \lesssim g $.
%On the other hand, we write $f\sim g$ when the asymptotic profile of~$f$ is described by~$g$, in an appropriate sense (for instance, a pointwise estimate as~$\xii\to0$ or an estimate in a functional space as~$t\to\infty$).
\end{notation}
\begin{notation}
We denote by~$\hat f=\mathfrak{F} f$ or~$\hat
f(t,\cdot)=\mathfrak{F} f(t,\cdot)$ the partial
Fourier transform, with respect to the space variable~$x$, of a
tempered distribution $S'(\mathbb{R}^n)$ or of a function, in the appropriate
distributional or functional sense and its inverse  transform
by~$\mathfrak{F}^{-1}$.
\end{notation}
\begin{notation}
By~$L^p=L^p(\R^n)$, $p\in[1,\infty]$, we denote the space of
measurable functions~$f$ such that~$|f|^p$ has finite integral
over~$\R^n$, if~$p\in[1,\infty)$, or has finite essential supremum
over~$\R^n$ if~$p=\infty$. We denote by~$W^{m,p}$, $m\in\N$, the
space of~$L^p$ functions with weak derivatives up to the~$m$-th
order in~$L^p$. We denote by~${H}^s(\R^n)$ and $\dot{H}^s(\R^n)$,
$s\geq0$, the spaces of tempered distributions $S'(\mathbb{R}^n)$
with~$(1+\xii^2)^{\frac{s}2}\,\hat u \in L^2$ and $\xii^{s}\,\hat u \in L^2$, respectively.
\end{notation}
\begin{notation}\label{DefLpqspaces}
By $L_p^q=L_p^q(\mathbb{R}^n)$ we denote the space of tempered distributions $T\in \mathcal S'(\mathbb{R}^n)$ such that~$T\ast f\in L^q$ for any~$f\in \mathcal S$, and
\[ \| T \ast f\|_{L^q} \leq C \|f\|_{L^p} \]
for all $f\in\mathcal S$ with a constant $C$, which is independent of $f$. In this case, the operator~$T\ast$ is extended by density from~$\mathcal S$ to~$L^p$.

By $M_p^q=M_p^q(\mathbb{R}^n)$, $p\leq q$, we
denote the set of Fourier transforms $\hat{T}$ of distributions $T
\in L_p^q$,
equipped with the norm % REMARK: f in S implies that m\hat f is a tempered distribution defined by its action \<m\hat f, \varphi\>=\<m,\hat f\varphi\>
\[ \|m\|_{M_p^q}:=\sup\big\{\|\mathfrak{F}^{-1}(m\mathfrak{F}(f))\|_{L^q}:f\in \mathcal{S}, \|f\|_{L^p}=1\big\},\]
and we set~$M_p=M_p^p$. A function~$m$ in $M_p^q$ is called a multiplier of type $(p,q)$. %REMARK: NOT ALL ELEMENTS OF M_p^q ARE FUNCTIONS. IF THEY ARE, THEY ARE CALLED MULTIPLIERS
\end{notation}
Now, let us introduce the Besov spaces (see~\cite{Tr}).
\begin{notation}\label{not:Besov}
We fix a nonnegative function $\psi\in\mathcal C^\infty$, having compact support in $\{ \xi \in \mathbb{R}^n : 2^{-1}\leq \xii\leq 2\}$, such that:
\begin{equation}\label{eq:partition}
\sum_{k=-\infty}^{+\infty} \psi_k(\xi)=1, \qquad \text{where~$\psi_k(\xi):=\psi(2^{-k} \xi)$.}
\end{equation}
(This property is easily obtained
if~$\psi(\xi)=\varphi(\xi/2)-\varphi(\xi)$, for
some~$\varphi\in\mathcal C^\infty$, with~$\varphi(\xi)=1$
for~$\xii\leq1/2$ and~$\varphi(\xi)=0$
if~$\xii\geq1$). For any~$p\in[1,\infty]$, we
define the Besov space
\[ B^0_{p,2} = \{ f\in\mathcal S': \ \forall k\in\Z, \ \mathfrak{F}^{-1}(\psi_k\hat f)\in L^p, \quad \|f\|_{B^0_{p,2}}<\infty \}, \]
where
\[ \|f\|_{B^0_{p,2}} = \|\mathfrak{F}^{-1}(\psi_k\hat f)\|_{\ell^2(L^p)} = \left(\sum_{k=-\infty}^{+\infty} \|\mathfrak{F}^{-1}(\psi_k\hat f)\|_{L^p}^2\right)^{\frac12}. \]
\end{notation}

We are interested in obtain
$L^p-L^q$ estimates to the solutions of the Cauchy problem
\eqref{mainlineq}. For this purpose it is used the following results about multipliers and special functions:
\begin{lemma}[Littman's Lemma] \label{ThmLittmanlemmapecher}
Suppose that the function $v = v(\eta) \in C_0^\infty$ with support
in $\displaystyle \{ \eta \in \mathbb{R}^n; \frac12 \leq |\eta| \leq
2 \}$ and the function $\omega = \omega(\eta) \in C^\infty$ in a
neighborhood of the support of $v$. Assume $\tau_0$ a large positive
number and the rank of the Hessian $H_\omega(\eta)$ satisfies rank
$H_\omega(\eta)\geq k$ on the support of $ v$. Then there exists an integer number $L$, such that for all
$\tau \geq \tau_0$ holds
    $$  \left\| \mathfrak{F}^{-1}_{\eta \rightarrow x} \left( e^{-i\tau \omega(\eta)} v(\eta) \right)   \right\|_{L^\infty(\mathbb{R}^n_x)} \lesssim (1+\tau)^{-\frac{k}{2}}\sum_{|\alpha|\leq L} \|D_{\eta}^\alpha v(\eta) \|_{L^\infty(\mathbb{R}^n_x)}. $$
    \end{lemma}
In  Proposition 2.5 of \cite{SS} one can find a simple proof of Lemma \ref{ThmLittmanlemmapecher}, from which it is easy to check that the statement remains valid whenever~$\omega$ and~$v$ depend on some parameter~$t$, provided that~$|\det H_\omega(t,\eta)|\geq c>0$, with~$c$ uniform with respect to~$t$.

In \cite{Sj} one can find the following result:
 %(that)\textcolor{green}{which} is  useful tool to derive $L^q-L^q$
%estimates.
\begin{proposition}[Berstein's inequality]\label{sigrid1}
Let $n\geq 1$ and $N>\frac{n}{2}$. If $f \in H^{N}$, then $\mathscr{F}^{-1}m\in L^1$ and there exists a constant $C>0$ such that
\[ \|\mathscr{F}^{-1}m\|_{L_1}\leq C \|f\|_{L^2}^{1-\frac{n}{2N}}\|D^Nf\|_{L^2}^{\frac{n}{2N}}.
\]
\end{proposition}
    In \cite{Bateman} one can find  the following  properties for Bessel and   Hankel functions:
     \begin{lemma}\label{hankel}
       The function
     \[ \Gamma_{\gamma}(\tau)=\tau^{-\gamma}J_{\gamma}(\tau),\]
     where  $J_{\gamma}(\tau)$ is the  Bessel function, is entire in $\gamma$ and $\tau$, in particular,
     \begin{eqnarray}
\label{Besselest}
 |J_{\gamma}(\tau)|  \lesssim \tau^{\gamma}, \quad 0<\tau <1.
 \end{eqnarray}

     The Weber's function $Y_{\gamma}(\tau)$ satisfies for every integer $n$
     \[ Y_{n}(\tau)= \frac2{\pi}J_{n}(\tau)\ln\tau + A_n(\tau),\]
     where $\tau^{n}A_{n}(\tau)$ is entire,  non-null for $\tau=0$  and
\begin{eqnarray}
\label{SecondBesselest}  |A_n(\tau)|  \lesssim
\tau^{-n}, \quad 0<\tau <1.
\end{eqnarray}
     The Hankel functions $H^{\pm}_{\gamma}=J_{\gamma}\pm iY_{\gamma}$ satisfy
    \[2(H^{\pm}_{\gamma})'(\tau)=H^{\pm}_{\gamma-1}(\tau)- H^{\pm}_{\gamma+1}(\tau), \quad  and \quad \tau (H^{\pm}_{\gamma})'(\tau)=\tau H^{\pm}_{\gamma-1}(\tau)-\gamma H^{\pm}_{\gamma}(\tau).\]
     Moreover,  $H^{\pm}_{\gamma}(\tau), \tau\geq K$ can be written as
\begin{equation}\label{HankelHigh}
 H^{\pm}_{\gamma}(\tau)=e^{\pm i\tau} a_{\gamma}^{\pm}(\tau),
  \end{equation}
  where  $a_{\gamma}^{\pm}(\tau)  \in S^{-\frac12}(K, \infty)$  is a classical symbol of order $-\frac12$.

  For small arguments $0<\tau\leq K <1$ we have
  \begin{equation}\label{HankelLow}
  |H^{\pm}_{\gamma}(\tau)|\lesssim \begin{cases}
  \tau^{-|\gamma|}, \qquad  if \ & \ \gamma\neq 0\\
  -\ln(\tau), \qquad  if & \gamma= 0.
  \end{cases}
  \end{equation}

  \end{lemma}

%%%%%%%%%%%%%%%%%%%%%%%%

\end{document}